%% file: clustergalois.tex
\documentclass[14pt]{amsart}
\usepackage{cases}
\usepackage{amsmath}
\usepackage{amsfonts}
\usepackage{bm}
\usepackage{amsfonts,amsmath,amssymb,amscd,bbm,amsthm,mathrsfs,dsfont}
\usepackage{mathrsfs}
\usepackage{pb-diagram}
\usepackage{color}
\usepackage{amssymb}
\usepackage{xypic}
\usepackage[all]{xy}
\usepackage{mathrsfs}
\usepackage{amsthm}
\usepackage[all]{xy}
\usepackage{epsfig}
\usepackage{graphics}
\usepackage{array}
\usepackage{graphicx} 
\usepackage{epstopdf}
\usepackage{float}
\usepackage{tikz}
\usepackage{tikz-cd}
\usetikzlibrary{arrows}
\usetikzlibrary{graphs}
\usepackage{tikz}
\usepackage{hyperref}
\usepackage{delarray}
\usepackage{CJK}

\usepackage{subfigure}
\usepackage{booktabs}
\usepackage{extarrows}
\usepackage{float}
\usepackage{rotating}

\newtheorem{thm}{Theorem}[section]
\newtheorem{lem}[thm]{Lemma}
\newtheorem{defi}[thm]{Definition}
\newtheorem{cor}[thm]{Corollary}
\newtheorem{prop}[thm]{Proposition}
\newtheorem{ex}[thm]{Example}
\newtheorem{rmk}[thm]{Remark}
\newtheorem{conj}[thm]{Conjecture}

\newtheorem{thma}{Theorem}

\newtheorem{propa}[thma]{Proposition}

\title[ On Galois theory of cluster algebras]
{ On Galois theory of cluster algebras: general and that from Riemann surfaces}

\author{Jinlei Dong \;\;\;and \;\;\; Fang Li  }
\address{Jinlei Dong
\newline
School of Mathematical Sciences,
Zhejiang University,
Yuhangtang Road 866,
Hangzhou, Zhejiang 310058,
China P.R. }
\email{jinleidong@zju.edu.cn}

\address{Fang Li
\newline School of Mathematical Sciences,
Zhejiang University,
Yuhangtang Road 866,
Hangzhou, Zhejiang 310058,
China P.R.}
\email{fangli@zju.edu.cn}

\date{\today}

\newcommand{\lra}{\longrightarrow}

\newcommand{\ra}{\rightarrow}
\newcommand{\sdp}{\times\kern-.2em\vrule height1.1ex depth-.05ex}
\newcommand{\epi}{\lra \kern-.8em\ra}

\makeatletter
\@addtoreset{equation}{section}
\makeatother

 \normalbaselines
\begin{document}
\renewcommand{\thefootnote}{\alph{footnote}}
\setcounter{footnote}{-1} \footnote{Project supported by the National Natural Science Foundation of China (No.12071422 and No.12131015). }

\renewcommand{\thefootnote}{\alph{footnote}}
\maketitle
\bigskip
	
\begin{abstract}
One of the key points in Galois theory via field extensions is to build up a correspondence between subfields of a field and subgroups of its automorphism group, so as to study fields via methods of groups. As an analogue of the Galois theory, we want to discuss the relations between cluster subalgebras of a cluster algebra and subgroups of its automorphism group and then set up the Galois-like method.

In the first part, we build up a Galois map from a skew-symmetrizable cluster algebra $\mathcal A$ to its cluster automorphism group, and introduce notions of Galois-like extensions and Galois extensions.  A necessary condition for Galois extensions of a cluster algebra $\mathcal A$ is given, which  is also a sufficient condition if $\mathcal A$ has a $\mathcal{D}$-stable basis or stable monomial basis with unique expression.
Some properties for Galois-like extensions are discussed. It is shown that two subgroups $H_1$ and $H_2$ of the automorphism group $\text{Aut}\mathcal A$ are conjugate to each other if and only if there exists $ f \in \text{Aut}\mathcal{A} $ and two Galois-like extension subalgebras $\mathcal A(\Sigma_1)$, $\mathcal A(\Sigma_2)$ corresponding to $H_1$ and $H_2$ such that $f$ is an isomorphism between $\mathcal A(\Sigma_1)$ and $\mathcal A(\Sigma_2)$.

In the second part, as the answers of two conjectures proposed in the first part, for a cluster algebra from a feasible surface, we prove that  Galois-like extension subalgebras corresponding to a subgroup of a cluster automorphism group have the same rank.
Moreover, it is shown that there are order-preserving reverse Galois maps for these cluster algebras. We also give examples of $\mathcal{D}$-stable bases and some discussions on the Galois inverse problem in this part.

\end{abstract}

\tableofcontents

\section{Introduction}

From Lagrange and Gauss to Abel and Galois, people have gradually come to understand symmetry as the code used by the creator to design the laws of all things.

The framework of the Galois theory was established by Galois via solving the problem of root solutions for higher order equations.

The usefulness of Galois groups provides the invariant under the meaning of symmetry in the structure of fields, which reflects the understanding of the ``laws" and is also the key to solving algebraic equations of higher order through the distribution law of roots.

The meaning of the method of Galois lies in its ability to extend to almost all mathematical objects. This is the significance of the automorphism group of any mathematical object.

The specific operation method is achieved through the correspondence between the sub-objects chains of mathematical objects and the subgroup chains of automorphism groups.

This paper aims to apply the method to cluster algebras so as to consider the symmetric invariant in the structure of cluster algebras.\vspace{2mm}

Cluster algebras come from many fields of mathematics, such as canonical bases and total positivity. It is a large class algebras that was axiomatically defined and introduced by Fomin and Zelevinsky in \cite{FZ02}.

The cluster algebras in \cite{FZ02, HuangLi23} were defined recursively: Firstly, give a seed consisting of exchange matrix and initial cluster variables, then introduce a formula called mutation to generate new cluster variables, and all cluster variables obtained by finite steps mutations generate an algebra called cluster algebras.

In \cite{FZ03} Fomin and Zelevinsky defined a strong cluster isomorphism between the cluster
algebras $ \mathcal{A}_{1} $ and $ \mathcal{A}_{2} $ which requires any seed in $ \mathcal{A}_{1} $ under the isomorphism to be a seed in $ \mathcal{A}_{2} $. And it can be easily proved that it is equivalent to the condition that the isomorphism can commute with mutations in a sense in \cite{ASS14}, which was called cluster isomorphism in \cite{ASS14}.

On the other hand, the author in \cite{HuangLiYang18} gave a definition of the morphism between cluster algebras $ \mathcal{A}_{1} $ and $ \mathcal{A}_{2} $ requiring it commutes with mutations in a sense. When the morphism is isomorphic, the definition is also cluster isomorphism mentioned above. Besides, if the morphism is monomorphic, $ \mathcal{A}_{1} $ is called a cluster subalgebra of the cluster algebra $ \mathcal{A}_{2} $.

	The author in \cite{HuangLiYang18} also gave a structure of all cluster subalgebras by mixing-type sub-seed. As for the cluster automorphisms of cluster algebras, the author in \cite{ASS14} obtained a family of cluster automorphism groups using the method of cluster categories and put forward a conjecture that these groups of cluster algebras from surfaces are isomorphic to mapping class groups, which has been proved in \cite{Gu11, BS15, BQiu15, BY18}.

One of the key points in Galois theory via field extensions is to build up a correspondence between subfields of a field and subgroups of its automorphism group, so as to study fields via methods of groups. As an analogue of the Galois theory, we want to discuss the relations between cluster subalgebras of a cluster algebra and subgroups of its automorphism group and then set up the Galois-like method.

In this paper, we try to build up the Galois theory of cluster algebras. We construct a Galois map from cluster subalgebras of a cluster algebra to subgroups of its Galois group and introduce the notions of Galois-like extensions and Galois extensions. In addition, we give a sufficient and necessary condition for Galois extensions. Finally, we discuss some properties of Galois-like extensions and confirm some conjectures proposed in Section 3 for the cluster algebras coming from surfaces with certain conditions. Moreover, we also discuss the Galois inverse problem for these cluster algebras.

As a supplement, we also introduce the notion of complement sub-seed which gives a way to embed the Galois group into a cluster automorphism group of a cluster subalgebra, and offer examples for Galois maps.

The article is organized as follows:

In Section 2, we introduce definitions and some preparation knowledge of cluster algebras, cluster morphisms, cluster subalgebras, and mapping class groups of marked surfaces.

In Section 3, we firstly introduce the Galois map from cluster subalgebras of a cluster algebra to their fixed subgroups which are called Galois group, and introduce the notion of complement sub-seed which gives a way to embed the Galois group into a cluster automorphism group of a cluster subalgebra.

Suppose $ H_{1}, H_{2} \in \mathit{2}^{\text{\em Aut}\mathcal{A}} \backslash \text{ker}\phi $ for a cluster algebra $ \mathcal{A} $ and $ \mathcal{M}_{sub}^{H_{i}},i=1,2 $ are their corresponding sets of Galois-like extension subalgebras, we have the following theorem:

\begin{thma}[Theorem \ref{galad}]
	$ H_{1} $ is conjugate to $ H_{2} $ if and only if there exist $ f \in \text{\em Aut}\mathcal{A} $, $ \mathcal{A}(\Sigma_{1}) \in \mathcal{M}_{sub}^{H_{1}}$, $ \mathcal{A}(\Sigma_{2}) \in \mathcal{M}_{sub}^{H_{2}} $ such that the restriction of $ f $ is an isomorphism between $ \mathcal{A}(\Sigma_{1}) $ and $ \mathcal{A}(\Sigma_{2}) $.
\end{thma}

Moreover, we introduce the notions of Galois-like extensions and Galois extensions and give a necessary condition for Galois extensions of cluster algebras, which is also a sufficient condition if cluster algebras have $\mathcal{D}$-stable bases or stable monomial bases with unique expression.

\begin{thma}[Theorem \ref{crit}]
		Let $\mathcal{A}=\mathcal{A}(\Sigma) $ be a cluster algebra and $ H $ be a subgroup of the cluster automorphism group of $\mathcal{A}(\Sigma)$. Suppose $ \mathcal{A}(\Sigma_{1}) $ is maximal in $ \mathcal{A}^{H} $ as a cluster subalgebra of $ \mathcal{A}$.

		(i)\; If $ \mathcal{A}^{H} = \mathcal{A}(\Sigma_{1}) $, then $ |Hz| = + \infty $ for any $ z \in \mathcal{A} \backslash \mathcal{A}(\Sigma_{1}) $ which can be represented as Laurent polynomials of some cluster with positive coefficients.
	
		(ii)\; Conversely, $ \mathcal{A}^{H} = \mathcal{A}(\Sigma_{1}) $ holds if either (a)   $\mathcal A$ has a stable monomial basis with the unique expression  and $ |Hz| = + \infty $ for any cluster variable $ z \in \mathcal{A} \backslash \mathcal{A}(\Sigma_{1}) $ or
		(b)   $\mathcal A$ has a $\mathcal{D}$-stable basis  $ |Hz| = + \infty $ for any cluster variable or quasi cluster variable $ z \in \mathcal{A} \backslash \mathcal{A}(\Sigma_{1}) $.
\end{thma}

Unfortunately, $ \mathcal{A}(\Sigma)^{H} $ is not a cluster subalgebra in many cases.
Thus we discuss Galois-like extensions. That is, $ \mathcal{A}(\Sigma_{1})\subset \mathcal{A}(\Sigma) $ satisfies that $ \mathcal{A}(\Sigma_{1}) $ is maximal in $ \mathcal{A}(\Sigma)^{H} $ as a cluster subalgebra of $ \mathcal{A}(\Sigma) $ and $ \text{Gal}_{\Sigma_{1}}\mathcal{A}(\Sigma) =H$.

In Section 4, we introduce the relation between mapping class groups and the corresponding cluster automorphism groups and give some examples of $\mathcal{D}$-stable bases for cluster algebras from the so-called feasible surfaces, which is given in Definition \ref{def-feasiblesurface}.

Furthermore, we give the Galois map and reverse Galois for cluster algebras from feasible surfaces as a partial confirmation of Conjecture \ref{ssq}.

Suppose $ (S,M) $ be a feasible surface and $ \mathcal{A}(S,M) $ is the cluster algebra without coefficients corresponding to $ (S,M) $, we have:

\begin{thma}[Theorem \ref{c1}]
	Let $ (S, M) $ be a feasible surface, $ \mathcal{A}(S, M) $ its cluster algebra without coefficient, and a subgroup $ H \le \text{\em Aut}\mathcal{A}(S, M) $. If $ \mathcal{A}(\Sigma_{i}), i=1, 2 $ are maximal in $ \mathcal{A}(S,M)^{H} $ as cluster subalgebras of $ \mathcal{A}(S,M) $, then the ranks of $ \mathcal{A}(\Sigma_{i}), i=1, 2 $ are the same.
\end{thma}
\begin{thma}[Theorem \ref{c2}]
	Let $ H \le \text{\em Aut}\mathcal{A}(S,M) $, $ \mathcal{A}(\Sigma_{1}) $, $ \mathcal{A}(\Sigma_{2}) \in \mathcal{M}^{H}_{sub} $, $ H_{1}, H_{2},...,H_{s} \in \mathit{2}^{\text{\em Aut}\mathcal{A}(S,M)} \backslash \text{ker}\phi $ and $ H_{1} \le H_{2} \le \cdots \le H_s$. Then for any $ \mathcal{A}(\Sigma_{i}) \in \mathcal{M}^{H_{i}}_{sub}, 1 \leq i \leq s $, there exist $ \mathcal{A}(\Sigma_{j}) \in \mathcal{M}^{H_{j}}_{sub}$ for $ j=1,..., i-1, i+1, ..., s $ such that
	$$ \mathcal{A}(\Sigma_{1}^{(d)}) \ge \cdots \ge \mathcal{A}(\Sigma_{i}^{(d)}) \ge \cdots \ge \mathcal{A}(\Sigma_{s}^{(d)}). $$
	That is, there is an order-preserving reverse Galois map for $ \mathcal{A}(S, M) $.\\
	In particular, if $ H_{1} < H_{2} <\cdots < H_{s} $ and any $ f \in H_{s} $ is orientation-preserved, then we have
	\[ \mathcal{A}(\Sigma_{1}) > \cdots > \mathcal{A}(\Sigma_{i}) > \cdots > \mathcal{A}(\Sigma_{s}). \]
\end{thma}

Here the reduced cluster subalgebras $\mathcal{A}(\Sigma_{i}^{(d)})$ and the  orientation-preserved cluster automorphism are defined in Section \ref{Galois-like-extension-def-subsection} and \ref{surface-conj-subsection}, respectively.

Besides, we also discuss the Galois inverse problem for cluster algebras from feasible surfaces.

\begin{propa}[Proposition \ref{prop-Galois-inverse-problem}]
		 Assume $ \mathcal{A}(S, M) $ is a cluster algebra without coefficient from a feasible surface $(S, M)$. Let $ C $ be a subsurface of $ S $ such that there is  no connected component of  $ U= \overline{S\backslash C}  $ which is a closed annulus, a closed disk, or a closed once-punctured disk and there is a connected component of $U$ which i different from a digon, a closed once-punctured annulus and a closed twice-punctured disk, where $\overline{S\backslash C}$ means the closure of $S\backslash C$. Then we have the following statements:
	
	(i)\; $ \rho'^{-1}(\text{\em Mod}_{\bowtie}(C,C\cap M)) $ is a Galois group.

 (ii)\;	Conversely, for a Galois group $ H  $ whose any element $ f $ is orientation-preserved, there is a subsurface $C$ such that $ H=\rho'^{-1}(\text{\em Mod}_{\bowtie}(C,C\cap M)) $.
	
\end{propa}

Here the definition of subsurfaces is given in Section \ref{surface-conj-subsection}, $\text{\em Mod}_{\bowtie}(C,C\cap M)$ is a subgroup of the mapping class group of $(C,C\cap M)$ which is also given in Section \ref{surface-conj-subsection} and $\rho'$ is the isomorphism between the mapping class group and cluster automorphism group, see Theorem \ref{mcg}.

\section{Preliminaries on cluster algebras and Riemann oriented surfaces}

In this section, we introduce some background materials of cluster algebras, the cluster algebras from bordered surfaces with marked points, and their mapping class groups.

Let $ \mathcal{F} $ be the field of rational polynomials over $ \mathbb{Q} $ in $ m $ independent variables, the \textbf{cluster algebra} $ \mathcal{A} $ is a $ R- $subalgebra of $ \mathcal{F} $, where $ R=\mathbb{Q}[x_{n+1},...,x_{m}] $, the ring of polynomials of the frozen variables.

\begin{defi}
	A seed in $ \mathcal{F} $ is a pair $ (\textbf{\em x}_{ex},\textbf{\em x}_{fr},\widetilde{B}) $ where
	
	(1) $ \textbf{\em x}_{ex}=(x_{1},...,x_{n}) , \textbf{\em x}_{fr}=(x_{n+1},...,x_{m})$ and $ x_{1},..., x_{m} $ are algebraically independent in $ \mathcal{F} $;
	
	(2) $ \widetilde{B} $ is an $ m \times n $ extended skew-symmetrizable integer matrix, i.e., there exists a diagonal non-negative integer matrix $ D $ such that $ (DB)^{\top}=-DB $, where $ B $ is the upper $ n \times n $ submatrix of $ \widetilde{B}$, called the \textbf{exchange matrix}; $\widetilde{B} $ is called the \textbf{extended exchange matrix} and $ D $ is called the\textbf{ skew-symmetrizer}. Meanwhile, the lower $ n \times n $ submatrix of $ \widetilde{B}$, denoted as $C$, is called the {\bf $C$-matrix}.
	
	(3) $\textbf{\em x}= \textbf{\em x}_{ex} \cup \textbf{\em x}_{fr} $ is called the \textbf{cluster} of the seed $ (\textbf{\em x}_{ex},\textbf{\em x}_{fr},\widetilde{B}) $. The elements of $ \textbf{\em x}_{ex} $ are called \textbf{exchange variables} and the elements of $ \textbf{\em x}_{fr} $ are called \textbf{frozen variables}. The exchange variables and frozen variables are collectively called \textbf{cluster variables}. The monomials consisting of cluster variables (or exchange variables) in a cluster are called \textbf{extended cluster monomials} (or \textbf{cluster monomials}).
\end{defi}

We say a cluster algebra to be {\bf with coefficients}, if the set of frozen variables in the cluster algebra is not empty, otherwise we say the cluster algebra to be {\bf without coefficient}.
\begin{defi}
	Let $ (\textbf{\em x}_{ex},\textbf{\em x}_{fr},\widetilde{B}) $ be a seed in $ \mathcal{F} $, $ k \in \{1,...,n \} $ The\textbf{ seed mutation} $ \mu_{k} $ transforms $(\textbf{\em x}_{ex},\textbf{\em x}_{fr},\widetilde{B}) $ to the new seed $ \mu_{k}(\textbf{\em x}_{ex},\textbf{\em x}_{fr},\widetilde{B}) = (\textbf{\em x}_{ex}',\textbf{\em x}_{fr}',\widetilde{B'})$ defined as follows:
	
	(1) $ \widetilde{B'}=(b_{ij}')_{m \times n} $, where
	$$ b_{ij}'= \left\{
	\begin{array}{ll}
		-b_{ij} & \text{if } i=k \text{ or } j=k; \\
		b_{ij}+ sgn(b_{ik})max(b_{ik}b_{kj},0) & \text{otherwise}.
	\end{array}
	\right. $$
	
	(2) $ (\textbf{\em x}_{ex}',\textbf{\em x}_{fr}')=(x_{1}',...,x_{m}') $, where
	$$ x_{i}'= \left\{
	\begin{array}{ll}
		x_{i} & \text{if } i \neq k ; \\
		x_{i}^{-1}(\prod_{b_{jk} > 0}x^{b_{jk}}_{j} + \prod_{b_{jk} <0 }x^{-b_{jk}}_{j} ) & \text{otherwise}.
	\end{array}
	\right. $$
\end{defi}

In the sequel we also denote by $ \mu_{x_{k}}= \mu_{k} $ and $ b_{x_{i}x_{j}}=b_{ij} $, for $ x_{k},x_{i} \in \textbf{x}, \, x_{j} \in \textbf{x}_{ex} $.

Let $ \mathcal{X} $ be the set of all cluster variables obtained from the initial seed $ \Sigma $ by finite steps of mutations, the {\bf cluster algebra} $ \mathcal{A} $ is defined as the $ \mathbb{Q}- $subalgebra of $ \mathcal{F} $ generated by $ \mathcal{X} $. The {\bf rank} of $\mathcal A$ is defined as the number of exchange variables in any seed $\Sigma$.

Note that if $ D $ is the skew-symmetrizer of $ \widetilde{B} $, it is also the skew-symmetrizer of $ \widetilde{B'} $.

In a cluster algebra, any exchange variable can be expressed in terms of any given cluster as a Laurent polynomial with coefficients in $ R $, which is called the \textbf{Laurent phenomenon } in the cluster algebra.
It is proved in \cite{GHKK18} that the coefficients in the Laurent polynomial mentioned above are positive.

Let $\mathbb{T}_n $ be the $n$-regular tree whose $n$ edges emanating from the same vertex are labelled bijectively by $1,...,n$. We assign a seed to each vertex of $\mathbb{T}_n$ such that if two vertices are connected by an edge labelled $k$, then the seeds assigned to them are obtained from each other by the mutation at direction $k$. This assignment is called a \textbf{cluster pattern}.

Let $ \mathcal{S} $ be a cluster pattern, $ \mathcal{L}(t)= \mathbb{Q}[x_{1;t}^{\pm 1}, ..., x_{m;t}^{\pm 1}]$, where $ t \in \mathbb{T}_{n} $ and $ x_{1;t}, ..., x_{m;t} $ are all cluster variables in the seed of $ \mathcal{S} $ corresponding to $ t $. The \textbf{upper cluster algebra} associated with $ \mathcal{S} $ is defined as \[ \mathcal{U}=\bigcap_{t \in \mathbb{T}_{n}}\mathcal{L}(t) . \]

Let $ \mathcal{A}(\Sigma) $ be a cluster algebra, where initial seed $ \Sigma=(\textbf{x}_{ex},\textbf{x}_{fr},\widetilde{B}_{m \times n} ) $, $ \textbf{x}_{ex}=\{ x_{1},...,x_{n} \} $, $\textbf{x}_{fr}=\{x_{n+1},...,x_{m} \} $. We denote by
\begin{equation}\label{exmatrix}
	\widetilde{B}=
	\left(
	\begin{array}{cc}
		B_{n \times n}\\
		C_{(m-n) \times n}
	\end{array}
	\right).
\end{equation}

A matrix is called \textbf{column sign-coherent} if any two nonzero entries of this matrix in the same column have the same sign.

Following the notions given above, $ C_{(m-n) \times n} $ is called \textbf{uniformly column sign-coherent} with respect to the exchange matrix $ B_{n\times n} $ if $ C_{(m-n) \times n} $ is column sign-coherent and after any sequence of mutations, the lower $ (m-n) \times n $ submatrix is column sign-coherent.

A \textbf{cluster algebra with principal coefficients} is a cluster algebra $ \mathcal{A}(\Sigma) $ with initial seed $ \Sigma=(\textbf{x}_{ex},\textbf{x}_{fr}, \widetilde{B}_{2n \times n} )$, where $ \textbf{x}_{ex}=\{x_{1},x_{2},...,x_{n} \} $, $ \textbf{x}_{fr}= \{ x_{n+1},x_{n+2},...,x_{2n} \}$ and the lower $ n \times n $ submatrix $B $ of $ \widetilde{B} $ is the identity matrix.

It was proved in \cite{GHKK18} that $C$-matrices of a skew-symmetrizable cluster algebra $\mathcal A$ with principal coefficients are always uniformly column sign-coherent.

For a skew-symmetrizable cluster algebra $ \mathcal{A}(\Sigma)$ with principal coefficients, there is a $ \mathbb{Z}^{n} $-grading structure which is defined by
\[ deg(x_{i})= \textbf{e}_{i}, \; deg(x_{n+j}) = -\textbf{b}_{j}, \, i,j=1,...,n, \]
where $ \{e_{1}, ..., e_{n}\} $ is the standard basis for $ \mathbb{Z}^{n} $ and $ \textbf{b}_{j} $ is the $ j$-th column of the upper $n\times n$ submatrix $B$ of $ \widetilde{B} $. It was proved in \cite{FZ07} that Laurent polynomials of exchange variables in $ \mathcal{A}(\Sigma) $ for $ \textbf{x}=\textbf{x}_{ex}\sqcup \textbf{x}_{fr} $ are homogeneous with respect to this $ \mathbb{Z}^{n} $-grading, and the $ \mathbb{Z}^{n} $-grading degree of an exchange variable is called its \textbf{$ g $-vector}. For a cluster $ \textbf{x}' $, we denote by $ G^{\widetilde{B}}_{\textbf{x}'} $ the matrix whose columns consisting of $ g$-vectors of all exchange variables in $ \textbf{x}' $, called the \textbf{$ G$-matrix} of the cluster $ \textbf{x}' $. Trivially, $ G^{\widetilde{B}}_{\textbf{x}'} $ is an $n\times n$ matrix.

The \textbf{$ F $-polynomial }of an exchange variable in the principal coefficients cluster algebra $ \mathcal{A}(\Sigma) $ is the polynomial obtained by specializing all of initial exchange variables to 1 in the Laurent polynomial of this exchange variable for the initial cluster. If the $ F $-polynomials of two exchange variables are the same, then these two exchange variables are the same, see \cite{LiPan22}.

For an exchange variable $ x' $ and a cluster $ \textbf{x}=\textbf{x}_{ex} \sqcup \textbf{x}_{fr} $ in an arbitrary cluster algebra $ \mathcal{A} $ with $ \textbf{x}_{ex}= \{ x_{1},...,x_{n} \} $ and $ \textbf{x}_{fr}=\{ x_{n+1},...,x_{m} \} $, we can represent $ x'$ as
$$ \frac{P(\textbf{x})}{x_{1}^{d_{1}} \cdots x_{n}^{d_{n}} },$$
where $ P(\textbf{x}) $ is a polynomial of $ \textbf{x} $ which is not divided by any cluster variable in $ \textbf{x}$. The \textbf{denominator vector} {\bf d} of $ x' $ for $ \textbf{x} $ is defined as {\bf d}$=(d_{1},...,d_{n})^{\top}$. In fact, the component $ d_{i} $ of {\bf d} corresponding to $ x_{i} $ is independent of the cluster $ \textbf{x} $, see \cite{CaoLi20a}. Thus we define the compatibility degree $ d(x_{i},x') $ of $ x' $ with respect to $ x_{i} $ as $ d_{i} $.

Let $ \mathcal{A}_{1} $ and $ \mathcal{A}_{2} $ be cluster algebras with coefficients. A cluster morphism between them is defined as follows:

\begin{defi}\cite{HuangLiYang18}
	Let $ \Sigma_{1}, \Sigma_{2} $ are seeds of cluster algebras $ \mathcal{A}_{1} $ and $ \mathcal{A}_{2} $ respectively, where $ \Sigma_{1}=(\textbf{\em x}_{ex,1},\textbf{\em x}_{fr,1},$ $\widetilde{B}_{1,m_{1} \times n_{1}}), $ $ \Sigma_{2}=(\textbf{\em x}_{ex,2},\textbf{\em x}_{fr,2},\widetilde{B}_{2,m_{2} \times n_{2}}) $, $ \textbf{\em x}_{ex,i}=\{ x_{i,1},x_{i,2},...,x_{i,n_{i}} \},i=1,2 $ are exchange variables and $ \textbf{\em x}_{fr,i}= \{ x_{i,n+1},x_{i,n+2},...,$ $x_{i,m_{i}} \}, i=1,2 $ are frozen variables. Suppose $ f: \mathcal{A}_{1}\rightarrow \mathcal{A}_{2} $ is a morphism of algebras, if it satisfies the following conditions:
	
	(1) $ f(\textbf{\em x}_{ex,1}\sqcup \textbf{\em x}_{fr,1}) \subseteq \textbf{\em x}_{ex,2}\sqcup \textbf{\em x}_{fr,2}\sqcup \mathbb{Z} $;
	
	(2) $ f(\textbf{\em x}_{ex,1}) \subseteq \textbf{\em x}_{ex,2} \sqcup \mathbb{Z} $;
	
	(3) If $ (y_{1},y_{2},...,y_{s}) $ is a $(f, \Sigma_{1}, \Sigma_{2}) $ bi-admissible sequence, i.e., $ y_{i}$ is an exchange variable in $ \mu_{y_{i-1}}...\mu_{y_{1}}(\Sigma_{1}) $ and $ f(y_{i}) $ is an exchange variable in $ \mu_{f(y_{i-1})}...$ $ \mu_{f(y_{1})}(f(\Sigma_{1})) $, $ i=1,...,s $, assume $ \mu_{y_{0}}(\Sigma_{1})=\Sigma_{1} $ and $ \mu_{f(y_{0})}(f(\Sigma_{1}))=f(\Sigma_{1}) $, we have
	\[ f(\mu_{y_{s}}...\mu_{y_{1}}(y))=\mu_{f(y_{s})}...\mu_{f(y_{1})}(f(y)), \]
	then we call $ f $ a \textbf{cluster morphism} bewteen $ \mathcal{A}_{1} $ and $ \mathcal{A}_{2} $.
\end{defi}

If a cluster morphism $ f $ between $ \mathcal{A}_{1} $ and $ \mathcal{A}_{2} $ is monomorphic (epimorphic, isomorphic) as a morphism of algebras, we say it to be cluster monomorphic (epimorphic, isomorphic).

Let $ \mathcal{A}(\Sigma) $ be a cluster algebra, where initial seed $ \Sigma=(\textbf{x}_{ex},\textbf{x}_{fr},\widetilde{B}_{m \times n} ) $, $ \textbf{x}_{ex}=\{ x_{1},...,x_{n} \} $, $\textbf{x}_{fr}=\{x_{n+1},...,x_{m} \} $, $ f \in \text{Aut} \mathcal{A}(\Sigma) $. We denote by
\begin{equation}\label{exmatrix}
	\widetilde{B}=
	\left(
	\begin{array}{cc}
		B_{n \times n}\\
		C_{(m-n) \times n}
	\end{array}
	\right),
	\widetilde{B}_{f}=
	\left(
	\begin{array}{cc}
		B_{f,n \times n}\\
		C_{f,(m-n) \times n}
	\end{array}
	\right),
\end{equation}
where $ \widetilde{B}_{f} $ is the skew-symmetrizable matrix corresponding to the cluster $ f(\textbf{x}_{ex}\cup \textbf{x}_{fr}) $, which can be obtained by relabelling indexes of the extended exchange matrix corresponding to $ f(\textbf{x}_{ex}\cup \textbf{x}_{fr}) $ in a seed in $ \mathcal{A}(\Sigma) $, and the $ i $ row (or $ j $ column ) corresponds to $ f(x_{i}) $ (or $ f(x_{j}) $). As in \cite{ASS14}, we have
\[
\left(
\begin{array}{cc}
	B\\
	C
\end{array}
\right)
=
\left(
\begin{array}{cc}
	B_{f}\\
	C_{f}
\end{array}
\right)
\text{or}
\left(
\begin{array}{cc}
	B\\
	C
\end{array}
\right)
=
-\left(
\begin{array}{cc}
	B_{f}\\
	C_{f}
\end{array}
\right).
\]
If the former holds, we call $ f $ a \textbf{direct cluster automorphism} and denote by the subgroup consisting of direct cluster automorphisms $ {\rm Aut}^{+}\mathcal{A}(\Sigma) $, otherwise $ f $ is called an \textbf{inverse cluster automorphism}.

\begin{defi}
	Let $\mathcal{A}_{2} $ and $ \mathcal{A}_{1} $ be cluster algebras such that $ \mathcal{A}_{2} \subset \mathcal{A}_{1} $. If the embedding from $ \mathcal{A}_{2} $ to $ \mathcal{A}_{1} $ is a cluster monomorphism, we say $ \mathcal{A}_{2} $ to be a \textbf{cluster subalgebra} of $ \mathcal{A}_{1} $.
\end{defi}

In \cite{HuangLiYang18} the cluster subalgebra was defined as the preimage of a monomorphic cluster morphism. But in this article, we require that the cluster subalgebra be a subalgebra of a cluster algebra which is different from \cite{HuangLiYang18}, since we only care about the cluster subalgebras which are fixed under the action of subgroups of cluster automorphism groups.

In fact, for a cluster algebra $ \mathcal{A} $ and its cluster subalgebra $ \mathcal{A}(\Sigma') $, the initial seed $ \Sigma' $ can be obtained by a sequence of operations on a seed belonging to $ \mathcal{A}$, see \cite{HuangLiYang18}, which is given as follows.

\begin{defi}\cite{HuangLiYang18}
	Let $ \Sigma=(\textbf{\em x}_{ex},\textbf{\em x}_{fr},\widetilde{B}_{m \times n}) $ be a seed of a cluster algebra $\mathcal A$, Assume $I_{0} $ and $ I_{1} $ are subsets of $ \textbf{\em x}_{ex} \sqcup \textbf{\em x}_{fr} $, such that $ I_{0} \subseteq \textbf{\em x}_{ex}, I_{1} \subseteq \textbf{\em x}_{ex}\sqcup \textbf{\em x}_{fr} $ and $ I_{0} \cap I_{1} = \emptyset $. Denote by $ \textbf{\em x}_{ex}^{\sharp}=\textbf{\em x}_{ex} \backslash( I_{0}\cup I_{1}),\textbf{\em x}_{fr}^{\sharp}=(\textbf{\em x}_{fr} \cup I_{0}) \backslash I_{1}$, and $\widetilde{B}^{\sharp}=(b_{ij}^{\sharp}) $ is a matrix with $ \sharp (\textbf{\em x}_{ex}^{\sharp} \sqcup \textbf{\em x}_{fr}^{\sharp}) \times \sharp \textbf{\em x}_{ex}^{\sharp} $ order, such that $ b^{\sharp}_{xy}=b_{xy}$, for $x\in \textbf{\em x}_{ex}^{\sharp} \sqcup \textbf{\em x}_{fr}^{\sharp}, y \in \textbf{\em x}_{ex}^{\sharp} $. Then we can define a new seed $\Sigma_{I_{0},I_{1}} = (\textbf{\em x}_{ex}^{\sharp},\textbf{\em x}_{fr}^{\sharp},\widetilde{B}^{\sharp}) $, which is called a \textbf{mixing-type sub-seed} or $ (I_{0},I_{1}) $-type sub-seed of the seed $ \Sigma $.
\end{defi}

The initial seed of a cluster subalgebra of a cluster algebra $ \mathcal{A} $ is a $ (I_{0}, I_{1})$-type sub-seed of some seed of $ \mathcal{A} $. Formally, we have the following theorem.

\begin{thm}\cite{HuangLiYang18}\label{subseed} A cluster algebra
	$ \mathcal{A}' $ is a cluster subalgebra of a cluster algebra $ \mathcal{A} $ if and only if there exists a seed $ \Sigma'$ belonging to $ \mathcal{A}'$ such that $ \Sigma'=\Sigma_{I_{0},I_{1}} $ for some seed $ \Sigma=(\textbf{\em x}_{ex},\textbf{\em x}_{fr},\widetilde{B}) $ of $ \mathcal{A} $ satisfying $ b_{xy}=0$, for $ x \in I_{1}$ and $ y \in \textbf{\em x}_{ex} \backslash ( I_{0} \cup I_{1} )$.
\end{thm}

Thus from now on we call a cluster algebra $ \mathcal{A}(\Sigma') $ a cluster subalgebra of a cluster algebra $\mathcal A= \mathcal{A}(\Sigma) $, if $ \Sigma' $ is a mixing-type sub-seed of $ \Sigma $ satisfying the condition of Theorem \ref{subseed}.

A \textbf{marked surface} is a pair $ (S, M) $, where $ S $ is a Riemann surface and $ M $ consists of finite points of $ S $ called {\bf marked points} of $ S $ such that each boundary component contains at least one marked point. We also call $ (S, M) $ a {\bf surface} for short. The points in $ M \backslash \partial S $ are called {\bf punctures}, where $ \partial S $ is the boundary of $ S $.

It should be noted that all surfaces $(S, M)$ mentioned in this article are connected oriented Riemann surfaces $ S$ with finite boundary components. Besides, following \cite{FST08}, we require all surfaces to be not one of the following:

{\em A sphere with 1, 2, or 3 marked points, an unpunctured disk with 1,2 or 3 marked points on the boundaries, or the once-punctured disk with 1 marked point on the boundaries.}

Cluster algebras from surfaces are a particular type of cluster algebras whose exchange matrices are determined by a triangulation of surfaces.

 A curve $ \gamma $ in $ S $ is a continuous map from the unit interval to $ S $. We can assign a natural orientation on $\gamma$ from $ \gamma(0) $ to $\gamma(1)$. The curve with the reverse orientation is denoted by $\gamma^{-1}$. 	

\begin{defi}\label{arc}
	An \textbf{arc} $ \gamma $ in $ (S,M) $ is a curve (up to isotopy) in $ S $ such that
	
	(i) the endpoints of $\gamma $ lie in $ M $;
	
	(ii) $ \gamma $ does not intersect with boundary of $ S $ except endpoints;
	
	(iii) $ \gamma $ does not cut out an unpunctured monogon or an unpunctured digon;
	
	(iv) $ \gamma $ does not intersect with itself.
\end{defi}

If a curve $ \gamma $ satisfies (i), (ii), (iii) of Definition \ref{arc}, then $ \gamma $ is called a \textbf{generalized arc}. If a curve $\alpha$ is isotopic to a segment of the boundary of $S$ such that endpoints are in $M$, then $\alpha$ is called a {\bf boundary arc} of $S$.

Two arcs in $ (S, M) $ are called \textbf{compatible} if there are curves in their respective isotopy classes that do not intersect. A maximal collection of distinct pairwise compatible arcs is called an \textbf{ideal triangulation}.

\begin{defi}
	An \textbf{tagged arc} is an arc in which each endpoint has been tagged plain or notched, so that the following conditions are satisfied:
	
	(i) the arc does not cut out a once-punctured monogon;
	
	(ii) all endpoints lying on the boundaries are tagged plain;
	
	(iii) if its endpoints are the same, they are tagged in the same way.
\end{defi}
Two tagged arcs $ \alpha $ and $ \beta $ are called \textbf{compatible} if the plain arcs $ \bar{\alpha} $, $ \bar{\beta} $ obtained by forgetting their taggings are compatible and satisfy

(a) if $ \bar{\alpha} = \bar{\beta}$, then at least one endpoint of $ \bar{\alpha} $ and $ \bar{\beta} $ is tagged in the same way;

(b) if $ \bar{\alpha} \neq \bar{\beta}$, but they have common endpoints, then they are tagged in the same way at common endpoints.

Similarly, a maximal collection of distinct pairwise compatible tagged arcs in $ (S, M) $ is called A \textbf{tagged triangulation} of $ (S, M) $.

In fact, each ideal triangulation (resp. tagged triangulation) has the same number of tagged arcs (resp. tagged arcs).
\begin{prop}\cite{FST08} \label{narc}
	Each ideal triangulation (resp. tagged triangulation) of $ (S, M) $ consists of
	$ n=6g+3b+3p+c-6 $
	arcs (resp. tagged arcs), where $ g $ is the genus of $ (S,M) $, $ b $ is the number of boundary components, $ p $ is the number of punctures, and $ c $ is the number of marked points on the boundaries.
\end{prop}

Suppose $ T $ is a triangulation (resp. tagged triangulation) of $ (S, M) $, and $ \gamma \in T $. Then there is another arc (resp. tagged arc) $ \gamma' $ such that $ T' =(T \backslash \{ \gamma \}) \cup \{ \gamma' \} $ is another triangulation (resp. tagged triangulation) of $ (S,M) $. And the way to obtain $ T' $ from $ T $ is called the \textbf{flip} at $ \gamma $.

Suppose $ \mathcal{A} $ is the cluster algebra whose exchange matrix is determined by a triangulation of $ (S, M) $.
\begin{thm}\cite{FST08}
	(1) If $ (S, M) $ is not a closed surface with one puncture, then there are bijections:
	\[ \{ \text{tagged arcs in } (S,M) \} \to \{\text{exchange variables in } \mathcal{A} \}, \]
	\[ \{ \text{tagged triangulations of } (S,M) \} to \{ c\text{lusters in } \mathcal{A} \}. \]
	
	(2) If $ (S, M) $ is a closed surface with one puncture, then there are bijections:
	\[ \{ \text{arcs in } (S,M) \} \to \{\text{exchange variables } in \mathcal{A} \}, \]
	\[ \{ \text{triangulations of } (S,M) \} \to \{ \text{clusters in } \mathcal{A} \}. \]
	
	And in these bijections, a flip at $ \gamma $ coincides with the mutation at the corresponding exchange variable $ x_{\gamma} $.
\end{thm}
Note that a loop is a closed curve with no self-intersection, for a loop $ \zeta $ surrounding a puncture $ p $ which is a tagged arc, its corresponding element $ x_{\zeta} $ in the corresponding cluster algebra $ \mathcal{A} $ satisfies $ x_{\zeta}=x_{\gamma}x_{\gamma^{p}} $, where $ \gamma $ is the arc lying in the monogon bounded by $ \zeta $ and $ \gamma^{p} $ is the tagged arc obtained from $ \gamma $ via changing the tagging at $ p $. Then $\zeta$ and $\gamma$  is said to form a \textbf{self-folding triangle}

\begin{defi}\cite{BY18}
	The \textbf{mapping class group} of a bordered surface with marked points $ (S,M) $ is defined by
	\[ \mathcal{MCG}(S,M) := \text{\em Homeo}^{+}(S,M)/\text{\em Homeo}^{+}_{0}(S,M) ,\]
	where $ \text{\em Homeo}^{+}(S,M) $ is the group of orientation-preserving homeomorphisms such that $ g(M)=M $ for any $ g \in \text{\em Homeo}^{+}(S,M) $, and $ \text{\em Homeo}^{+}_{0}(S,M) $ is its subgroup of homeomorphisms isotopic to the identity.
	Let $ P $ be the set of punctures of $(S,M) $, the \textbf{tagged mapping class group} of $ (S,M) $ is defined to be the semidirect product
	\[ \mathcal{MCG}_{\bowtie}(S,M) := \mathcal{MCG}(S,M) \ltimes \mathbb{Z}_{2}^{P}, \]
	if $ (g_{1},p_{1}),(g_{2},p_{2}) \in \mathcal{MCG}_{\bowtie}(S,M)$, $ (g_{1},p_{1})(g_{2},p_{2}):=(g_{1}g_{2},p_{1}g_{1}(p_{2})) $.
\end{defi}

We can also define an unsigned version of the mapping class group. Denote
by $ \text{Homeo}(S, M) $ the group of homeomorphisms of $ S $ that takes $ M $ to itself. Define
\[ \mathcal{MCG}^{\pm}(S,M) := \text{Homeo}(S,M)/\text{Homeo}^{+}_{0}(S,M) ; \]
\[ \mathcal{MCG}^{\pm}_{\bowtie}(S,M) := \mathcal{MCG}^{\pm}(S,M) \ltimes \mathbb{Z}_{2}^{P}.\]

In fact, the tagged mapping class group of a bordered surface with marked points is isomorphic to the cluster automorphism group of the cluster algebra corresponding to this surface, which is introduced in Section 4.

\section{Cluster automorphism groups, extensions and $\mathcal{D}$-stable bases for cluster algebras}

In this section, we introduce some notions of Galois theory for cluster algebras and give a sufficient and necessary condition for Galois extensions of cluster algebras with $\mathcal{D}$-stable bases. And, we discuss some properties of Galois-like extensions. Moreover, the conjugate of subgroups of a cluster automorphism group is given.

As a supplement, we also introduce the notion of complement sub-seed which gives a way to embed the Galois group into a cluster automorphism group of a cluster subalgebra, and offer an example for cluster automorphism groups embedded into $ S_{n} \rtimes \mathbb{Z}_{2} $.

\subsection{Galois group, Galois map and cluster automorphism groups embedded into $ S_{n} \rtimes \mathbb{Z}_{2} $ }\label{3.1}\quad

If there is a cluster subalgebra of a cluster algebra $ \mathcal{A} $, whose initial seed is a mixing-type sub-seed of a seed $ \Sigma $ of $ \mathcal{A} $, then we can define the complement of this sub-seed, which is also a mixing-type sub-seed.

\begin{defi}\label{defcomp} Suppose $ \mathcal{A}(\Sigma') $ is a cluster subalgebra of a cluster algebra $\mathcal A= \mathcal{A}(\Sigma) $, where $ \Sigma=(\textbf{\em x}_{ex},\textbf{\em x}_{fr},\widetilde{B}_{m\times n})$ with a mixing-type sub-seed $ \Sigma'=(\textbf{\em x}_{ex}',\textbf{\em x}_{fr}',\widetilde{B'}_{m'\times n'})$.
	The {\bf complement sub-seed} of $ \Sigma' $ is defined as
	\[ \Sigma''=(\textbf{\em x}_{ex}'',\textbf{\em x}_{fr}'',\widetilde{B''}), \]
	where $ \textbf{\em x}_{ex}''=\textbf{\em x}_{ex} \backslash (\textbf{\em x}_{ex}' \cup \textbf{\em x}_{fr}') $, $\textbf{\em x}_{fr}''=\textbf{\em x}_{fr} \cup (\textbf{\em x}_{ex} \cap \textbf{\em x}_{fr}') $, and $ \widetilde{B''}$ is a submatrix of $ \widetilde{B} $ with columns labelled by $ \textbf{\em x}_{ex}'' \subset \textbf{\em x}_{ex} $ and lows labelled by $ \textbf{\em x}_{ex}''\sqcup \textbf{\em x}_{fr}'' \subset \textbf{\em x}_{ex}\sqcup \textbf{\em x}_{fr} $.
\end{defi}

Excluding the frozen variables in $ \Sigma' $, the exchange variables of $ \Sigma $ is the disjoint union of the exchange variables in $ \Sigma' $ and the exchange variables in $ \Sigma'' $, i.e., $\textbf{x}_{ex} \backslash \textbf{x}_{fr}'= \textbf{x}_{ex}' \sqcup \textbf{x}_{ex}'' $.

\begin{prop}\label{asub}
	For a seed $ \Sigma $ and its a mixing-type sub-seed $ \Sigma'$ as given in Definition \ref{defcomp}, the complement
	$ \Sigma'' $ is a mixing-type sub-seed of $ \Sigma $ and $ \mathcal{A}(\Sigma'') $ is a cluster subalgebra of $ \mathcal{A}(\Sigma) $.
\end{prop}

\begin{proof}
	It is obvious that $ \Sigma'' $ is a $ (I_{0},I_{1}) $-type sub-seed of $ \Sigma $. According to Theorem \ref{subseed}, it is sufficient to prove $ b_{xy}=0$, for all $ x \in I_{1}$, and $ y \in \textbf{x}_{ex} \backslash I_{0} \cup I_{1} $.
	
	We can see $ I_{0}=\textbf{x}_{ex} \cap \textbf{x}_{fr}', I_{1}=\textbf{x}_{ex}' $. Because $ I_{1}=\textbf{x}_{ex}' $, the condition $ b_{xy}=0$, for all $ x \in I_{1}$, and $ y \in \textbf{x}_{ex} \backslash I_{0} \cup I_{1} $ is equivalent to $ b_{yx}=0$, for all $ x \in I_{1}$, and $ y \in \textbf{x}_{ex} \backslash I_{0} \cup I_{1} $.
	
	On the other hand, because $ \Sigma' $ is $ (I_{0}',I_{1}') $-type sub-seed of $ \Sigma $, we have $ I_{0}'= (\textbf{x}_{ex} \backslash \textbf{x}_{ex}') \cap \textbf{x}_{fr}', I_{1}'=\textbf{x}_{ex} \sqcup \textbf{x}_{fr} \backslash \textbf{x}_{ex}' \sqcup \textbf{x}_{fr}' $. Since $ \mathcal{A}(\Sigma') $ is a cluster subalgebra of $ \mathcal{A}(\Sigma_{1}) $, we have $ b_{xy}=0$, for all $ x \in I_{1}'$, and $ y \in \textbf{x}_{ex} \backslash I_{0}' \cup I_{1}' $.
	
	$ \textbf{x}_{ex} \backslash I_{0}' \cup I_{1}'= (\textbf{x}_{ex} \backslash I_{0}') \cap (\textbf{x}_{ex} \backslash I_{1}')=\textbf{x}_{ex}' $ and $ \textbf{x}_{ex} \backslash I_{0} \cup I_{1} =(\textbf{x}_{ex} \backslash I_{0}) \cap (\textbf{x}_{ex} \backslash I_{1}) $, where $ \textbf{x}_{ex} \backslash I_{0}=\textbf{x}_{ex} \backslash \textbf{x}_{fr}' \subset I_{1}' $ and $ \textbf{x}_{ex} \backslash I_{1}= \textbf{x}_{ex} \backslash \textbf{x}_{ex}' \subset I_{1}' $, thus $ b_{yx}=0$, for all $ x \in I_{1}$, and $ y \in \textbf{x}_{ex} \backslash I_{0} \cup I_{1} $. So $ \mathcal{A}(\Sigma'') $ is a cluster subalgebra of the cluster algebra $ \mathcal{A}(\Sigma) $.
	
\end{proof}
Following this proposition, we call the subalgebra $ \mathcal{A}(\Sigma'') $ the \textbf{complement cluster subalgebra} of $ \mathcal{A}(\Sigma)$ associated to $ \mathcal{A}(\Sigma')$.

Let $ \mathcal{A}(\Sigma') $ be a cluster subalgebra of a cluster algebra $ \mathcal{A} $. Suppose $ \text{Aut} \mathcal{A} $ is the cluster automorphism group of $ \mathcal{A} $, then we can define the fixed subgroup of $ \text{Aut}\mathcal{A} $ over $ \mathcal{A}(\Sigma') $ as follows:
\[ \text{Gal}_{\Sigma'}\mathcal{A} =\{f \in \text{Aut}\mathcal{A} , f|_{\mathcal{A}(\Sigma')}=id_{\mathcal{A}(\Sigma')}\}, \]
which is called the {\bf Galois group} of $\mathcal A$ over $ \mathcal{A}(\Sigma') $.

If the cluster subalgebra $ \mathcal{A}(\Sigma') $ is trivial, i.e., $ \mathcal{A}(\Sigma') = \mathbb{Q} $ whose mixing-type sub-seed $\Sigma'$ is just the empty set $ \emptyset $, then $\text{Aut}\mathcal{A}(\Sigma) = \text{Gal}_{\emptyset}\mathcal{A}(\Sigma)$.

Now we show that the Galois group $ \text{ Gal}_{\Sigma'}\mathcal{A}(\Sigma) $ can be embedded into $ \text{Aut}\mathcal{A}(\Sigma'') $.
let $ \mathcal{A}(\Sigma'') $ be the complement cluster subalgebra of $ \mathcal{A}(\Sigma) $ associated to $ \mathcal{A}(\Sigma') $, we have:

\begin{prop} \label{comembed} Following the notions given above, we have the following statements.
	
	(i)\; The map
	\[
	\pi:\; \text{\em Gal}_{\Sigma'}\mathcal{A}(\Sigma) \rightarrow \text{\em Gal}_{\textbf{\em x}_{fr}'}\mathcal{A}(\Sigma'') \;\;\;\text{via}\;\;\;
	f \mapsto f|_{\mathcal{A}(\Sigma'')}
	\]
	is a monomorphism, where $ \textbf{\em x}_{fr}' $ is the set of frozen variables in $ \Sigma' $, which is considered as a mixing-type sub-seed of the cluster subalgebra $ \mathbb{Q}[\textbf{\em x}_{fr}'][\emptyset] $ without exchange variables.
	
	(ii)\; Suppose $ f' \in \text{\em Gal}_{\textbf{\em x}_{fr}'}\mathcal{A}(\Sigma'')$. Then
	
	(a) $ \bar{\textbf{\em x}}= \textbf{\em x}'_{ex}\cup f'(\textbf{\em x}''_{ex}\cup \textbf{\em x}''_{fr}) $ is a cluster of $ \mathcal{A}(\Sigma) $, and
	
	(b)
	$ f' \in \text{\em Im}\pi $ if and only if the submatrix $ D^{0}_{f} $ of extended exchange matrix of $ \bar{\textbf{\em x}} $ corresponding to $ f'(\textbf{\em x}''_{fr}) \times f'(\textbf{\em x}'_{fr} \cap \textbf{\em x}_{ex}) $ is the same as the submatrix $ D^{0} $of extended exchange matrix of $ \textbf{\em x}_{ex}\cup \textbf{\em x}_{fr} $ corresponding to $ \textbf{\em x}''_{fr} \times \textbf{\em x}'_{fr} \cap \textbf{\em x}_{ex} $.
	
	In particular, denoting by
	\[
	\widetilde{B}''=
	\left(
	\begin{array}{c}
		B''\\
		C''
	\end{array}
	\right),
	\]
	the extended exchange matrix of $ \Sigma'' $, where $ B'' $ is the exchange matrix of $ \Sigma'' $, if $ C'' $ is uniformly column sign-coherent with respect to $ B'' $, then $ \text{\em Gal}_{\textbf{\em x}''_{fr}}\mathcal{A}(\Sigma'') \subset \text{\em Im}\pi $.
\end{prop}

\begin{proof}
	Suppose $ \Sigma=(\textbf{x}_{ex},\textbf{x}_{fr},\widetilde{B}) $, $ \Sigma'=(\textbf{x}'_{ex},\textbf{x}'_{fr},\widetilde{B}') $ and $ \Sigma''=(\textbf{x}''_{ex},\textbf{x}''_{fr},\widetilde{B}'') $ with the extended exchange matrices $
	\widetilde{B}=(b_{xy})_{x\in \textbf{x}_{ex}\cup \textbf{x}_{fr}, y \in \textbf{x}_{ex}} $, $ \widetilde{B}'=(b_{xy})_{x\in \textbf{x}'_{ex}\cup \textbf{x}'_{fr}, y \in \textbf{x}'_{ex}} $ and $ \widetilde{B}''=(b_{xy})_{x\in \textbf{x}''_{ex}\cup \textbf{x}''_{fr}, y \in \textbf{x}''_{ex}} $. According to Definition \ref{defcomp}, we have $ \textbf{x}_{ex}= \textbf{x}'_{ex} \cup \textbf{x}''_{ex} \cup (\textbf{x}'_{fr}\cap \textbf{x}_{ex}) $, $ \textbf{x}'_{ex} \cap \textbf{x}''_{ex}= \emptyset $ and $ \textbf{x}''_{fr}=(\textbf{x}'_{fr} \cap \textbf{x}_{ex}) \cup \textbf{x}_{fr} $. We have
	\[
	\widetilde{B}=
	\left(
	\begin{array}{ccc}
		B_{\textbf{x}'_{ex} \times \textbf{x}'_{ex}}^{1} & A_{\textbf{x}'_{ex} \times \textbf{x}''_{ex}}^{2} & B_{\textbf{x}'_{ex} \times \textbf{x}'_{fr} \cap \textbf{x}_{ex}}^{3}\\
		A_{\textbf{x}''_{ex} \times \textbf{x}'_{ex}}^{1} & C_{\textbf{x}''_{ex} \times \textbf{x}''_{ex}}^{1} & C_{\textbf{x}''_{ex} \times \textbf{x}'_{fr} \cap \textbf{x}_{ex}}^{3}\\
		B_{\textbf{x}'_{fr} \cap \textbf{x}_{ex} \times \textbf{x}'_{ex}}^{2} & C_{\textbf{x}'_{fr} \cap \textbf{x}_{ex} \times \textbf{x}''_{ex}}^{2} & D_{\textbf{x}'_{fr} \cap \textbf{x}_{ex} \times \textbf{x}'_{fr} \cap \textbf{x}_{ex}}^{1}\\
		\hline
		B_{\textbf{x}'_{fr} \cap \textbf{x}_{fr} \times \textbf{x}'_{ex}}^{4} & C_{\textbf{x}'_{fr} \cap \textbf{x}_{fr} \times \textbf{x}''_{ex}}^{4} & D_{\textbf{x}'_{fr} \cap \textbf{x}_{fr} \times \textbf{x}'_{fr} \cap \textbf{x}_{ex}}^{2} \\
		A_{\textbf{x}_{fr} \backslash \textbf{x}'_{fr} \times \textbf{x}'_{ex}}^{3} & C_{\textbf{x}_{fr} \backslash \textbf{x}'_{fr} \times \textbf{x}''_{ex}}^{5} & D_{\textbf{x}_{fr} \backslash \textbf{x}'_{fr} \times \textbf{x}'_{fr} \cap \textbf{x}_{ex}}^{3}
	\end{array}
	\right).
	\]
	Suppose $ \textbf{x}_{ex}=\{x_{1}, ..., x_{n}\} $, $ \textbf{x}_{fr}=\{x_{n+1}, ..., x_{m} \} $ and $ D=diag(d_{x_{1}}, ..., d_{x_{n}}) $ is a skew-symmetrizer of $ \widetilde{B} $.
	
	(i)\;
	Firstly, it is easy to check that the map is a well-defined homomorphism. Besides, if $ g, g' \in \text{Gal}_{\Sigma'}\mathcal{A}(\Sigma) $ and $ g \neq g' $, then there exists $ z \in \Sigma \backslash \Sigma' \subset \Sigma'' $ such that $ g(z) \neq g'(z) $. Thus $ g|_{\mathcal{A}(\Sigma'')}(z) \neq g'|_{\mathcal{A}(\Sigma'')}(z) $, which means the map $ \pi $ is a monomorphism.
	
	(ii)\; (a)\; Let $ f $ be the endomorphism of the rational polynomials field over $ \textbf{x}_{ex}\cup \textbf{x}_{fr} $, which is an extension of $ f' $ defined by
	\begin{equation}\label{keepmut}
		f(x)=\left\{ \begin{array}{ll}
			x, & \text{if} \;\; x \in \textbf{x}'_{ex}, \\
			f'(x), & \text{if} \;\; x \in \textbf{x}''_{ex} \cup \textbf{x}''_{fr}.
		\end{array} \right.
	\end{equation}
	Since $ f' \in \text{Gal}_{\textbf{x}_{fr}'}\mathcal{A}(\Sigma'')$, $ f'( \textbf{x}''_{ex} \cup \textbf{x}''_{fr}) $ is a cluster of $ \mathcal{A}(\Sigma'') $, and hence we can make a sequence of mutations on exchange variables of $ \textbf{x}''_{ex}$ and the obtained successors from $ \textbf{x}''_{ex} \cup \textbf{x}''_{fr} $ to $ f'( \textbf{x}''_{ex} \cup \textbf{x}''_{fr}) $. Besides, $\mathcal{A}(\Sigma'') $ is a cluster subalgebra of $ \mathcal{A}(\Sigma) $, the embedding from $ \mathcal{A}(\Sigma'') $ to $ \mathcal{A}(\Sigma) $ is a cluster monomorphism. Thus by (\ref{keepmut}) we have $ \bar{\textbf{x}}= \textbf{x}'_{ex}\cup f'(\textbf{x}''_{ex}\cup \textbf{x}''_{fr}) =f(\textbf{x}'_{ex} \cup \textbf{x}''_{ex} \cup \textbf{x}''_{fr}) = f(\textbf{x}_{ex}\cup \textbf{x}_{fr}) $ and $ \bar{\textbf{x}} $ can be obtained by
	doing the sequence of mutations mentioned above from $ \textbf{x}_{ex}\cup \textbf{x}_{fr}$. Hence $f(\textbf{x}_{ex}\cup \textbf{x}_{fr}) $ is a cluster of $ \mathcal{A}(\Sigma) $.
	
	(b)\; The algebraic independence of clusters guarantees $ f $ is well-defined. Thus $ f' \in \text{Im}\pi $ if and only if $ f $ is a cluster automorphism of $ \mathcal{A}(\Sigma) $, which is equivalent to the extended exchange matrices $ \widetilde{B}, \widetilde{B}_{f}=(b_{xy}^{f})_{f(\textbf{x}_{ex} \cup \textbf{x}_{fr})\times f(\textbf{x}_{ex})} $ associated with $ \textbf{x}_{ex} \cup \textbf{x}_{fr} $ and $ f(\textbf{x}_{ex} \cup \textbf{x}_{fr})$ are the same, according to \cite{ASS14}.
	
	``only if":\;
	If $ f' \in \text{Im}\pi $, then $ D^{0}_{f} $ of the extended exchange matrix of $ \textbf{x}' $ corresponding to $ f'(\textbf{x}''_{fr}) \times f'(\textbf{x}'_{fr} \cap \textbf{x}_{ex}) $ is the same as the submatrix $ D^{0} $ of the extended exchange matrix of $ \textbf{x}_{ex}\cup \textbf{x}_{fr} $ corresponding to $ \textbf{x}''_{fr} \times (\textbf{x}'_{fr} \cap \textbf{x}_{ex}) $.
	
	``if":\;
	Suppose $ A^{i}_{f}, B^{j}_{f}, C^{k}_{f}, D^{l}_{f}$ correspond to $ A^{i}, B^{j}, C^{k}, D^{l} $ for $ i, l=1, 2, 3$; $j=1, 2, 3, 4$; $k=1, 2, 3, 4, 5$.
	Since $ \mathcal{A}(\Sigma') $ is a cluster subalgebra, The submatrices $ A^{1} $ and $ A^{3} $ of $ \widetilde{B} $ are equal to zero, according to Theorem \ref{subseed}. The skew-symmetrizablity of $ \widetilde{B} $ guarantees that $ A^{2}=0 $.
	
	Since $ f' \in \text{Gal}_{\textbf{x}_{fr}'}\mathcal{A}(\Sigma'') $ and $ \mathcal{A}(\Sigma'') $ is a cluster subalgebra of $ \mathcal{A}(\Sigma) $, there is a sequence of mutation from $ \Sigma $ at indexes in $ \textbf{x}''_{ex} $ to $ f(\Sigma
	) $, which is denoted by $ \mu_{f'} $, and $ \mu_{f'} $ also can be consider as a sequence of mutation from $ \Sigma'' $ to $ f'(\Sigma'') $. Since $ A^{i}=0 $, $ \mu_{f} $ does not change the submatrices $ A^{i}, B^{j} $, $ i=1, 2, 3; j=1, 2, 4 $. Thus we have $A^{i}=A^{i}_{f}, B^{j}=B^{j}_{f} $.
	
	Besides, $ f' \in \text{Gal}_{\textbf{x}_{fr}'}\mathcal{A}(\Sigma'') $ means $ \widetilde{B}'' $ and $ \widetilde{B}''_{f}=(b_{xy}^{f})_{f'(\textbf{x}''_{ex} \cup \textbf{x}''_{fr}) \times f'(\textbf{x}''_{ex})} $ are the same, i.e., $ b_{xy}=b_{f'(x)f'(y)}, x \in \textbf{x}''_{ex} \cup \textbf{x}''_{fr}, y\in \textbf{x}''_{ex}$. Note that $ \widetilde{B}''_{f} $ is the submatrix of $ \widetilde{B}_{f} $ corresponding to $ C_{1}^{f}, C_{2}^{f}, C_{4}^{f}, C_{5}^{f} $. Thus we have $ C_{i}=C_{i}^{f}, i=1, 2, 4, 5$. Because $ D $ is a skew-symmetrizer of $ \widetilde{B} $, $ D_{f}=(d_{f(x_{1})}, ..., d_{f(x_{n})}) $ is a skew-symmetrizer of $ \widetilde{B}_{f} $. Thus we have $ d_{x}b_{xy}=-d_{y}b_{yx} $ and $ d_{f(x)}b_{f(x)f(y)}=-d_{f(y)}b_{f(y)f(x)} $ for entries $ b_{xy} $ belong to $ B^{2}, C^{2} $. Since $ f(x)=x $ for $ x \in \textbf{x}'_{fr} \cap \textbf{x}_{ex} $, we have
	\[ b_{yx}=-\frac{d_{x}}{d_{y}} b_{xy}=-\frac{d_{f(x)}}{d_{f(y)}} b_{f(x)f(y)} =b_{f(y)f(x)}, \]
	for entries $ b_{xy} $ belong to $ B^{2}, C^{2} $. Hence $ C^{3}=C^{3}_{f}, B^{3}=B^{3}_{f} $.
	
	The condition $ D^{0}=D^{0}_{f} $ means $ D^{i}=D^{i}_{f}, i=1, 2, 3 $. Thus we have $ \widetilde{B}=\widetilde{B}_{f} $ and hence $ f' \in \text{Im}\pi $.
	
	According to the definition of mutations of matrices, mutation at $ z \in \textbf{x}''_{ex} $, $ \mu_{z}(b_{xy})= b_{xy}+ sgn(b_{xz})max(b_{xz}b_{zy},0) $ for $ x \in \textbf{x}''_{fr}, y\in \textbf{x}'_{fr} \cap \textbf{x}_{ex} $. If $ C'' $, i.e., the submatrix consisting of $ C^{2}, C^{4}, C^{5} $, is uniformly column sign-coherent, then
	\[ \begin{array}{cl}
		\mu_{z}(b_{xy}) &= b_{xy}+ sgn(b_{xz})max(b_{xz}b_{zy},0) \\
		&=b_{xy} + sgn(b_{xz})max(-b_{xz}\frac{d_{y}}{d_{z}}b_{yz},0) \\
		&=b_{xy},
	\end{array} \]
	for $ b_{xy} $ belonging to $ D^{i}, i=1, 2, 3 $.
	Thus muatations at $ \textbf{x}''_{ex} $ do not change entries belong to $D^{i}, i=1, 2, 3 $ and hence for $ f' \in \text{Gal}_{\textbf{x}_{fr}''}\mathcal{A}(\Sigma'') $, we have $ \widetilde{B} =\widetilde{B}_{f} $, which means $ f \in \text{Im}\pi $.
	
\end{proof}

According to \cite{CaoLi19}, if $ C^{2}, C^{4}, C^{5} $ are nonnegative submatrix, then $ C'' $ is uniformly column sign-coherent, and hence we have $ \text{Gal}_{\textbf{x}_{fr}''}\mathcal{A}(\Sigma'') \subset \text{Im}\pi$.

Assume $ \mathcal{A}(\Sigma_{i}) $ is a cluster subalgebra of $ \mathcal{A}(\Sigma_{i-1}), i=1,2,...,s$ where $ \mathcal{A}(\Sigma_{s})=0 $ and $ \mathcal{A}(\Sigma_{0})= \mathcal{A}(\Sigma) $, then we have
\begin{equation} \label{subasq}
	\mathcal{A}(\Sigma_{0}) \geq \mathcal{A}(\Sigma_{1}) \geq ... \geq \mathcal{A}(\Sigma_{s});
\end{equation}
\begin{equation}\label{subgsq}
	\text{Gal}_{\Sigma_{0}}\mathcal{A}(\Sigma) \leq \text{Gal}_{\Sigma_{1}}\mathcal{A}(\Sigma) \leq ... \leq \text{Gal}_{\Sigma_{s}}\mathcal{A}(\Sigma). \end{equation}
The sequence (\ref{subgsq}) is a filtration of subgroups of the cluster automorphism group $\text{Aut}\mathcal{A}(\Sigma)$ which corresponds to the filtration of cluster subalgebras of $\mathcal A(\Sigma)$ in (\ref{subasq}). Thus some properties of the sequence of algebras may be reflected via the sequence (\ref{subgsq}) of subgroups.

We denote by:\\ (i) $ \mathit{2}^{\text{Aut}\mathcal{A}} $ the set of all subgroups of $\text{Aut}\mathcal{A}$,\\ (ii) $ \mathit{2}^{ \mathcal{A} } $ the set of cluster subalgebras of $\mathcal A$,\\ (iii) $ \mathit{2}^{\mathit{2}^{ \mathcal{A} }} $ the set of all sets consisting of some cluster subalgebras of $\mathcal A$.

Define a map
\begin{equation}\label{galmap}
	\xi:\;\; \mathit{2}^{ \mathcal{A} }\rightarrow \mathit{2}^{\text{Aut}\mathcal{A}}
	\text{ via}\;\;\; \mathcal A(\Sigma')\mapsto \text{Gal}_{\Sigma'}\mathcal A(\Sigma)
\end{equation}
where $\Sigma'$ is a mixing-type sub-seed of $\Sigma$. We call $\xi$ the {\bf Galois map} from cluster subalgebras of $\mathcal A(\Sigma)$ to subgroups of its cluster automorphism group $\text{Aut}\mathcal A(\Sigma)$.

By (\ref{subasq}) and (\ref{subgsq}), the Galois map $\xi$ sends a descending sequence in $\mathit{2}^{ \mathcal{A} }$ to a ascending sequence in $\mathit{2}^{\text{Aut}\mathcal{A}}$.

However, when cluster subalgebras $ \mathcal{A}(\Sigma_{1}) \lneqq \mathcal{A}(\Sigma_{2}) $ or even $ \mathcal{A}(\Sigma_{1}) $ and $ \mathcal{A}(\Sigma_{2}) $ do not include each other, it is still possible to hold that $ \text{Gal}_{\Sigma_{1}}\mathcal{A}(\Sigma)=\text{Gal}_{\Sigma_{2}}\mathcal{A}(\Sigma) $.

For example, the fixed subgroups of any cluster subalgebra of type $ A_{n}$ cluster algebras with exchange variables are trivial. And Proposition \ref{exam} also gives another example.

Before that, we need the following theorem and lemma.

\begin{thm}\cite{NZ12} \label{cg}
	Let $ \mathcal{A}(\Sigma) $ be a cluster algebra with principal coefficients about the initial seed $ \Sigma=(\textbf{\em x}_{ex},\textbf{\em x}_{fr},\widetilde{B}_{2n \times n} ) $ with $ C=I_n$ the lower $ n \times n $ submatrix of $ \widetilde{B}$, then
	\[ G^{\widetilde{B}}_{\textbf{\em x}'}=D((C^{\widetilde{B}}_{\textbf{\em x}'})^{-1})^{\top}D^{-1},\]
	where $ C^{\widetilde{B}}_{\textbf{\em x}'} $ is the lower $ n \times n $ submatrix of extended exchange matrix $ \widetilde{B}_{\textbf{\em x}'} $ of $ \textbf{\em x}' $ and D is the skew-symmetrizer.
\end{thm}

\begin{lem}\cite{CaoLi20a} \label{g}
	Let $ \mathcal{A}(\Sigma) $ be a cluster algebra with principal coefficients, suppose $ x_{1},x_{2} $ are exchange variables of $ \mathcal{A}(\Sigma) $. Then $ x_{1}=x_{2} $ if and only if their $ g $-vectors $ \textbf{g}(x_{1}) $ and $ \textbf{g}(x_{2}) $ are the same.
\end{lem}

For $ \sigma \in S_{n} $ is a permutation of $ n $, we can consider $ \sigma $ as a permutation matrix. Its action on $ n \times n $ matrix $ M_{n \times n} $ is the multiplication of matrices. Then we have the following proposition.

\begin{prop}\label{exam}
	Let $ \mathcal{A}(\Sigma) $ be a cluster algebra with principal coefficients, where $ \Sigma=(\textbf{\em x}_{ex},\textbf{\em x}_{fr},\widetilde{B}_{2n \times n} ) $. Then $ \text{\em Aut}\mathcal{A}(\Sigma) $ is a subgroup of $ S_{n} \rtimes \mathbb{Z}_{2} $ where $ S_{n} $ is the group of all permutations on the initial exchange variables.
	
\end{prop}

\begin{proof}
	According to \cite{ASS14}, cluster automorphism group is semiproduct of direct cluster automorphism group and $ \mathbb{Z}_{2} $, hence we only consider the direct cluster automorphism group  of $ \mathcal{A} $.
	For any $ f \in \text{Aut}^{+}\mathcal{A}(\Sigma) $, suppose $ G_{f} $ is the $ G $-matrix of $ f(\textbf{x}_{ex}) $ and $ \widetilde{B}_{f} $ is the skew-symmetrizable matrix corresponding to $ f(\textbf{x}_{ex}, \textbf{x}_{fr}) $, then we have
	\[
	\widetilde{B}= \left(
	\begin{array}{cc}
		B\\
		C
	\end{array}
	\right)
	=
	\left(
	\begin{array}{cc}
		B_{f}\\
		C_{f}
	\end{array}
	\right)
	=\widetilde{B}_{f}.
	\]
	According to Theorem \ref{cg}, since $ C=I $, after the action of some permutions, we have
	$$ (G_{f})\sigma_{1} =D((\sigma_{2}(C_{f})\sigma_{1})^{-1})^{\top}D^{-1}
	= D((\sigma_{2}\sigma_{1})D^{-1}.$$
	Thus
	\[ (G_{f})\sigma_{2}^{-1}=D((\sigma_{2}\sigma_{1})D^{-1}\sigma_{1}^{-1}\sigma_{2}^{-1}, \]
	The left side of the equality is a matrix with their entries belonging to $ \mathbb{Z} $, so is the right side. And the right side is a multiplication of diagonal matrices $ D $ and $(\sigma_{2}\sigma_{1})D^{-1}\sigma_{1}^{-1}\sigma_{2}^{-1}$, where the entries of $ (\sigma_{2}\sigma_{1})D^{-1}\sigma_{1}^{-1}\sigma_{2}^{-1} $ are reciprocal of the entries of $ D $, thus the right side of the equality is the identity matrix and so is the left side.
	
	According to Lemma \ref{g}, after the action of a permution, the elements of $ f(\textbf{x}_{ex}) $ is equal to the elements of $ \textbf{x}_{ex} $, thus $ \text{Aut}\mathcal{A}(\Sigma) $ is a subgroup of $ S_{n} \rtimes \mathbb{Z}_{2} $.
	
\end{proof}

This proposition says that despite $ \mathcal{A} $ has many cluster subalgebras, their fixed subgroups of $ \text{Aut}\mathcal{A} $ could also be the same. In fact, the cluster automorphism groups of cluster algebras with principal coefficients are finite, but their cluster subalgebras may be infinite, which means there exist infinite cluster subalgebras such that their fixed subgroups are the same.

So we should consider a special class of subalgebras and subgroups.

\subsection{Galois-like extensions, conjugate of subgroups and some conjectures}\label{Galois-like-extension-def-subsection}\quad

	For a cluster algebra $\mathcal A=\mathcal A(\Sigma)$, let $H$ be a subgroup of $ \text{Aut}\mathcal{A}$. Denote by
	\[ \mathcal{A}^{H}=\{ z \in \mathcal{A}\; |\; f(z)=z, \forall f \in H \}.\]
	In general, $\mathcal{A}^{H}$ maybe not a cluster subalgebra of $\mathcal A$. However, we have that for a cluster subalgebra $ \mathcal{A}(\Sigma_{1}) $, if $ H=\text{Gal}_{\Sigma_{1}}\mathcal{A} $ then $\mathcal{A}(\Sigma_{1})\subseteq\mathcal{A}^{H}\subseteq \mathcal{A}(\Sigma)$.
	
	\begin{rmk}\label{notsubcluster}
		If $ \mathcal{A} $ is a cluster algebra of finite type, for any cluster subalgebra $ \mathcal{A}(\Sigma_{1}) $ of $ \mathcal{A} $, $ \mathcal{A}(\Sigma_{1})\subsetneqq\mathcal{A}^{\text{\em Gal}_{\Sigma_{1}}\mathcal{A}}$. In fact, let $ z $ be the product of all cluster variables of $ \mathcal{A} $, then for any $ f \in \text{\em Aut}\mathcal{A} $, $ f(z) = z $, since $ f $ is a bijection among cluster variables. Hence $ z $ is not in any proper cluster subalgebra of $ \mathcal{A} $.
	\end{rmk}
	
	In general, it is not difficult to find a cluster subalgebra $ \mathcal{A}(\Sigma_{1}) $ of a cluster algebras $ \mathcal{A} $ satisfying the condition $ \mathcal{A}(\Sigma_{1})\subsetneqq\mathcal{A}^{\text{Gal}_{\Sigma_{1}}\mathcal{A}} $.
	
	Thus for a given subgroup $ H \leq\text{Aut}\mathcal{A} $, we can consider a maximal cluster subalgebra $ \mathcal{A}(\Sigma_{1}) $ of $ \mathcal{A} $ such that
	$ \mathcal{A}(\Sigma_{1}) \subseteq \mathcal{A}^{H}$.
	
	Motivated by this fact, for $ H $ a subgroup of $ \text{Aut}\mathcal{A} $, we define a set of some cluster subalgebras as follows:
	\[
	\mathcal{M}_{sub}^{H}:=
	\left\{ \mathcal{A}(\Sigma') \left|
	\begin{array}{ccc}
		\mathcal{A}(\Sigma')\; \text{ is a maximal cluster subalgebra} \; \text{ of }\; \mathcal A
		\; \text{ in } \mathcal{A}^{H} \;
		\text{ satisfying}\; \text{Gal}_{\Sigma'}\mathcal{A}=H
	\end{array}
	\right. \right\}.
	\]
	
	\begin{rmk} We will show Example \ref{ex6gon} in the sequel which gives us a maximal cluster subalgebra $ \mathcal{A}(\Sigma') $ in $ \mathcal{A}^{H} $ such that $ \text{\em Gal}_{\Sigma'}\mathcal{A} \neq H \leq \text{\em Aut}\mathcal{A} $. Hence in the definition of $\mathcal{M}_{sub}^{H}$ above, the condition $ \text{\em Gal}_{\Sigma'}\mathcal{A}=H $ is not trivial.
	\end{rmk}
	
	Trivially, $\mathcal{M}_{sub}^{H}\in\mathit{2}^{\mathit{2}^{ \mathcal{A} }}$. Then we can define a map
	$$ \phi: \mathit{2}^{\text{Aut}\mathcal{A} } \rightarrow \mathit{2}^{\mathit{2}^{ \mathcal{A} }}\; \text{via}\;
	H \mapsto \mathcal{M}_{sub}^{H}. $$
	And, define $$ \text{ker} \phi= \{ H \leq \text{Aut}\mathcal{A}\; |\; \phi(H)=\emptyset \}.$$
	which means
	$$ H\in \text{ker} \phi \;\; \text{if and only if}\;\; \mathcal{M}_{sub}^{H}=\varnothing\;\; \text{if and only if} \;\;\nexists \mathcal A(\Sigma')\; \text{such that}\; \text{Gal}_{\Sigma'}\mathcal{A}=H.$$
	
	Now we give an example of $ H \in \text{ker} \phi $ i.e.,  $\mathcal{M}_{sub}^{H}=\varnothing $.
	
	\begin{ex}
		Let $ S $ be a hexagon, $ \mathcal{A} $ the cluster algebra from $ S $, $ H $ the subgroup of the cluster automorphism group $\text{\em Aut}\mathcal{A}$ generated by the $ 60^{\circ} $ rotation. Then the only cluster subalgebra in $ \mathcal{A}^{H} $ is trivial. However, if the cluster subalgebra $\mathcal{A}' $ of $ \mathcal{A} $ is trivial, then its Galois group is all cluster automorphism group. Thus $ \mathcal{M}_{sub}^{H}=\emptyset $.
	\end{ex}

	Next, we have
	\begin{prop}
		The map $ \phi $ restricted to $ \mathit{2}^{\text{\em Aut}\mathcal{A}} \backslash \text{ker} \phi $ is injective.
	\end{prop}
	\begin{proof}
		Suppose $ H_{1}, H_{2} \in \mathit{2}^{\text{Aut}\mathcal{A}} \backslash \text{ker} \phi $, $ \phi(H_{1}) = \mathcal{M}^{H_{1}}_{sub} = \mathcal{M}^{H_{2}}_{sub}=\phi(H_{2}) $. Thus there exists a cluster subalgebra $ \mathcal{A}(\Sigma') \in \mathcal{M}^{H_{1}}_{sub} = \mathcal{M}^{H_{2}}_{sub} $, which means $ H_{1}=\text{Gal}_{\Sigma'}\mathcal{A}=H_{2} $.
	\end{proof}

	\begin{prop}
		For $ \xi $ defined in (\ref{galmap}), {\em Im}$\xi=\mathit{2}^{\text{Aut}\mathcal{A}} \backslash \text{ker} \phi $.
	\end{prop}
	\begin{proof}
		If $ H = \xi(\mathcal{A}(\Sigma')) $ for some cluster subalgebra $ \mathcal{A}(\Sigma') \subset \mathcal{A} $, then there exist a maximal cluster subalgebra between $ \mathcal{A}(\Sigma') $ and $ \mathcal{A}^{H} $ denoting by $\mathcal{A}(\Sigma'') $. It means $ \mathcal{A}(\Sigma'') \in \mathcal{M}_{sub}^{H} $ and hence $ H \in \mathit{2}^{\text{Aut}\mathcal{A}} \backslash \text{ker} \phi $.
		Conversely, if $ H \in \mathit{2}^{\text{Aut}\mathcal{A}} \backslash \text{ker} \phi $, then $ \xi(\mathcal{A}(\Sigma''))=H $ for $ \mathcal{A}(\Sigma'') \in \mathcal{M}_{sub}^{H} $.
		
	\end{proof}
	
	Refer to the notion of Galois extension of a field, we introduce the following notion:
	\begin{defi} For a cluster algebra $ \mathcal A=\mathcal{A}(\Sigma) $ and $ H \in \mathit{2}^{\text{\em Aut}\mathcal{A}} \backslash \text{ker} \phi $,
		
		(i)\; For a cluster subalgebra $ \mathcal{A}(\Sigma_{1}) \in \mathcal{M}_{sub}^{H} $, let $ H=\text{\em Gal}_{\Sigma_{1}}\mathcal{A} $. Then we call $\mathcal{A}(\Sigma_{1})\subseteq\mathcal{A}^{H}\subseteq \mathcal{A}(\Sigma)$ a {\bf Galois-like extension}, and $ \mathcal{A}(\Sigma_{1}) $ is called a \textbf{Galois-like extension subalgebra} of $\mathcal{A}(\Sigma)$.
		
		(ii)\; Moreover, following (i), we call $ \mathcal{A}(\Sigma_{1})\subseteq\mathcal{A}(\Sigma)$ a {\bf Galois extension}, if it holds that
		\begin{equation}\label{galext}
			\mathcal{A}^{H}=\mathcal{A}(\Sigma_{1}).
		\end{equation}
		In this case $ \mathcal{A}(\Sigma_{1}) $ is called an \textbf{Galois extension subalgebra} of $\mathcal{A}(\Sigma) $.
		
	\end{defi}
	
	If we consider a strictly descending sequence of Galois-like extension subalgebra, then the sequence of subgroups obtained similarly as (\ref{subgsq}) is strictly ascending.
	\begin{prop}\label{strictgsub}
		If there is a strictly descending sequence of Galois-like extension subalgebras of a cluster algebra $ \mathcal{A}$
		\[ \mathcal{A}(\Sigma_{1}) > \cdots > \mathcal{A}(\Sigma_{i}) > \cdots > \mathcal{A}(\Sigma_{s}), \]
		then there is a strictly ascending sequence of subgroups of $ \text{\em Aut}\mathcal{A}$:
		\[\text{\em Gal}_{\Sigma_{1}}\mathcal{A}< \cdots <\text{\em Gal}_{\Sigma_{i}}\mathcal{A} < \cdots <\text{\em Gal}_{\Sigma_{s}}\mathcal{A}.\]
	\end{prop}
	\begin{proof}
		There is a natural ascending sequence of subgroups:
		\[\text{Gal}_{\Sigma_{1}}\mathcal{A} \le \cdots \le \text{Gal}_{\Sigma_{i}}\mathcal{A} \le \cdots \le \text{Gal}_{\Sigma_{s}}\mathcal{A}. \]
		Since $ \mathcal{M}_{sub}^{H_{i}} \not= \mathcal{M}_{sub}^{H_{j}} $ induces $ H_{i} \neq H_{j} $, all subgroups in the above sequence are different. So, the result follows.
	\end{proof}

	\begin{prop}
		Suppose $ \mathcal{A}(\Sigma) $ is a cluster algebra and $ H \le \text{\em Aut}\mathcal{A}(\Sigma) $, if $ \mathcal{A}(\Sigma_{1}) $ is maximal in $ \mathcal{A}(\Sigma)^{H} $ as a cluster subalgebra of $ \mathcal{A}(\Sigma) $. Denoting $ H'= \text{\em Gal}_{\Sigma_{1}}\mathcal{A}(\Sigma) $. Then $ \mathcal{A}(\Sigma_{1})\subseteq\mathcal{A}^{H'}\subseteq \mathcal{A}(\Sigma) $ is a Galois-like extension.
	\end{prop}
	\begin{proof}
		Since $ \mathcal{A}(\Sigma_{1}) $ is maximal in $ \mathcal{A}(\Sigma)^{H'} \subset \mathcal{A}(\Sigma)^{H} $ as a cluster subalgebra of $ \mathcal{A}(\Sigma) $, we have $ \mathcal{A}(\Sigma_{1}) \in \mathcal{M}_{sub}^{H'} $.
		
	\end{proof}

	Being similar to the Galois theory in fields, we have the following theorem for cluster subalgebras of a cluster algebra and subgroups of its automorphism group.
	
	\begin{thm}\label{galad}
		Let $ H_{1}, H_{2} \in \mathit{2}^{\text{\em Aut}\mathcal{A}} \backslash \text{ker}\phi $, where $ \mathcal{A}$ is a cluster algebra. Then $ H_{1} $ is conjugate to $ H_{2} $ if and only if there exist $ f \in \text{\em Aut}\mathcal{A} $, $ \mathcal{A}(\Sigma_{1}) \in \mathcal{M}_{sub}^{H_{1}}$, $ \mathcal{A}(\Sigma_{2}) \in \mathcal{M}_{sub}^{H_{2}} $ such that the restriction of $ f $ is a cluster isomorphism between $ \mathcal{A}(\Sigma_{1}) $ and $ \mathcal{A}(\Sigma_{2}) $.
	\end{thm}
	\begin{proof}
		Suppose $ H_{1} $ is conjugate to $ H_{2} $, then there exists $ f \in \text{Aut}\mathcal{A} $ such that $ f^{-1}H_{2}f=H_{1} $ or $ fH_{1}f^{-1}=H_{2} $. If $ x \in \mathcal{A}(\Sigma_{1}) $, let $ \mathcal{A}(\Sigma_{1}) \in \mathcal{M}_{sub}^{H_{1}}$, then $ f(\mathcal{A}(\Sigma_{1}) ) $ is also a cluster subalgebra of $ \mathcal{A} $, which is denoted by $ \mathcal{A}(\Sigma_{2}) $. Since $ \mathcal{A}(\Sigma_{1}), \mathcal{A}(\Sigma_{2})$ are cluster subalgebras of $ \mathcal{A} $ and $ f \in \text{Aut}\mathcal{A} $, $ f $ commutes with mutations in $ \mathcal{A}(\Sigma_{1}) $ and $ \mathcal{A}(\Sigma_{2}) $. Thus $f $ is a cluster isomorphism between $ \mathcal{A}(\Sigma_{1}) $ and $ \mathcal{A}(\Sigma_{2}) $.
		
		The cluster subalgebra $ \mathcal{A}(\Sigma_{2}) $ consist of $ f(x), x \in \mathcal{A}(\Sigma_{1}) $, which is invariant under the action of $ H_{2}=fH_{1}f^{-1} $, thus $ H_{2} \leq \text{Gal}_{\Sigma_{2}}\mathcal{A} $. Besides, if $g\in\text{Gal}_{\Sigma_{2}}\mathcal{A}$, then $ f^{-1}(g(f(x)))=f^{-1}(f(x))=x, x \in \mathcal{A}(\Sigma_{1}) $, $ f^{-1}gf \in H_{1}=f^{-1}H_{2}f $. Thus $ g \in H_{2} $, which means $ H_{2} \geq \text{Gal}_{\Sigma_{2}}\mathcal{A} $. So $ H_{2} = \text{Gal}_{\Sigma_{2}}\mathcal{A} $.
		
		On the other hand, if there exists a cluster subalgebra $ \mathcal{A}(\Sigma_{3}) $ such that $ \mathcal{A}(\Sigma_{2}) \subsetneq \mathcal{A}(\Sigma_{3}) \subset \mathcal{A}^{H_{2}} $, then $ f^{-1}( \mathcal{A}(\Sigma_{3}) ) $ is the cluster subalgebra and $ \mathcal{A}(\Sigma_{1}) \subsetneq f^{-1}( \mathcal{A}(\Sigma_{3}) ) \subset \mathcal{A}^{H_{1}} $, which is contrary to the definition of $ \mathcal{A}(\Sigma_{1}) $. Thus $ \mathcal{A}(\Sigma_{2}) $ is a maximal cluster subalgebra in $ \mathcal{A}^{H_{2}} $ and $ \mathcal{A}(\Sigma_{2}) \in \mathcal{M}_{sub}^{H_{2}}$.
		
		Conversely, if there exist $ f \in \text{Aut}\mathcal{A} $, $ \mathcal{A}(\Sigma_{1}) \in \mathcal{M}_{sub}^{H_{1}}$, $ \mathcal{A}(\Sigma_{2}) \in \mathcal{M}_{sub}^{H_{2}} $ such that $ f $ is an isomorphism from $ \mathcal{A}(\Sigma_{1}) $ to $ \mathcal{A}(\Sigma_{2}) $. We can see that $ f^{-1}H_{2}f \leq H_{1} $. Besides, $ f^{-1} $ is also an isomorphism from $ \mathcal{A}(\Sigma_{2}) $ to $ \mathcal{A}(\Sigma_{1}) $, thus we have $ f^H_{1}f^{-1} \leq H_{2} $, which means $ f^{-1}H_{2}f \geq H_{1} $. So $ f^{-1}H_{2}f = H_{1} $.
		
	\end{proof}
	
	The maximal cluster subalgebras belonging to $ \mathcal{M}_{sub}^{H} $ in $ \mathcal{A}^{H} $ is not unique in general. Now we discuss some common properties of such maximal cluster algebras in $ \mathcal{M}_{sub}^{H} $ and then connect them with the method of Galois theory. We begin with some conjectures:
	
	\begin{conj}\label{srk}
		For a cluster algebra $ \mathcal{A} $ and a subgroup $H$ of $\text{\em Aut}\mathcal{A} $, let $ \mathcal{A}(\Sigma_{1}) $, $ \mathcal{A}(\Sigma_{2}) \in \mathcal{M}^{H}_{sub} $. Then the rank of $ \mathcal{A}(\Sigma_{1}) $ is the same as that of $ \mathcal{A}(\Sigma_{2}) $.
	\end{conj}
	
	Conversely, assume $ \mathcal{A}(\Sigma_{1}) $ and $ \mathcal{A}(\Sigma_{2}) $ are cluster subalgebras of $ \mathcal{A} $ with the same number of exchange variables, they may not always be in a same set $ \mathcal{M}_{sub}^{H} $ for some $ H \le \text{Aut}\mathcal{A}$.
	
	\begin{ex}\label{ex6gon}
		In fact, consider a hexagon with its vertices labelled by $ 1,...,6 $ clockwise, then the cluster subalgebra $ \mathcal{A}_{1} $ with a frozen variable corresponding to the diagonal $ \gamma_{14} $ is in $ \mathcal{M}_{sub}^{H_{1}} $, where $ H_{1} $ is generated by two reflections and a rotation fixing $ \gamma_{14} $. However, the cluster subalgebra $ \mathcal{A}_{2} $ with two frozen variables corresponding to the diagonals $ \gamma_{26}, \gamma_{35} $ is in $ \mathcal{M}_{sub}^{H_{2}} $, where $ H_{2} $ is generated by one of these reflections fixing $ \gamma_{26}, \gamma_{35} $. Furthermore, the numbers of exchange variables in these subalgebras are all zero.
	\end{ex}
	
	By the Galois map in (\ref{galmap}), give a totally ordered sequence of cluster subalgebras, and their corresponding fixed subgroups also form a totally ordered sequence. Conversely, give a totally ordered sequence of cluster subgroups, we wonder to ask if there is a totally ordered sequence of cluster subalgebras. This is our consideration in the following conjecture.
	
	Let $ \mathcal{A}(\Sigma) $ be a cluster algebra with initial extended exchange matrix $\widetilde{B}=(b_{xy}) $. We define $ \Sigma^{(d)} $ to be the mixing-type sub-seed of $ \Sigma $ obtained by removing any frozen variable $ y $ satisfying $ b_{yx}=0 $ for any exchange variable $ x $ in $ \Sigma $. Obviously, $\mathcal{A}(\Sigma^{(d)}) $ is a cluster subalgebra of $ \mathcal{A}(\Sigma) $. We call $ \Sigma^{(d)} $ and $\mathcal{A}(\Sigma^{(d)}) $ the {\bf reduced sub-seed} and {\bf reduced cluster subalgebra} of $ \Sigma $ and $ \mathcal{A}(\Sigma) $ respectively.
	
	\begin{conj}\label{ssq}
		Let $ H_{1}, H_{2},...,H_{s} \in \mathit{2}^{\text{\em Aut}\mathcal{A}} \backslash \text{ker}\phi $ and $ H_{1} \le H_{2} \le \cdots \le H_s$. Then for any $ \mathcal{A}(\Sigma_{i}) \in \mathcal{M}^{H_{i}}_{sub}, 1 \leq i \leq s $, there exist $ \mathcal{A}(\Sigma_{j}) \in \mathcal{M}^{H_{j}}_{sub}$ for $ j=1,..., i-1, i+1, ..., s $ such that
		$$ \mathcal{A}(\Sigma_{1}^{(d)}) \ge \cdots \ge \mathcal{A}(\Sigma_{i}^{(d)}) \ge \cdots \ge \mathcal{A}(\Sigma_{s}^{(d)}). $$
	\end{conj}
	
	Suppose Conjecture \ref{ssq} holds for a cluster algebra $\mathcal A$, then we can define a map such that
	\begin{equation}\label{reversemap}
		\zeta:\;\; \text{The set of ascending sequences in}\;\mathit{2}^{\text{Aut}\mathcal{A}}\rightarrow \text{The set of descending sequences in}\; \mathit{2}^{ \mathcal{A} }
	\end{equation}
	\[\text{ via}\;\;\; \zeta(H_i)= \mathcal{A}(\Sigma_{i}^{(d)}),\; i=1,\cdots,s.\]
	We call $\zeta$ the {\bf reverse Galois map} from $\mathit{2}^{\text{Aut}\mathcal A(\Sigma)}$ to $\mathit{2}^{\mathcal A(\Sigma)}$.
	
	Let $A \subset \mathit{2}^{\mathcal A(\Sigma)} $ and $G \subset \mathit{2}^{\text{Aut}\mathcal A(\Sigma)}$ be maximal subsets such that $\zeta$ and $\xi$ are mutual invertible maps between $A$ and $G$. Then we call $\xi$ a {\bf Galois correspondence} and $\zeta$ a {\bf reverse Galois correspondence} respectively between $A$ and $G$.

	In the next section, we confirm Conjecture \ref{ssq} and Conjecture \ref{srk} to be true for cluster algebras from surfaces with some conditions.

	\subsection{Galois extensions, stable monomial bases and $\mathcal{D}$-stable bases for cluster algebras}

\begin{defi}\label{def-stablemonomialbases}
	Let $ \mathcal{A} $ be a cluster algebra. If $W$ is a $\mathbb{Q}$-basis for $\mathcal{A}$ such that elements in $W$ are monomials of cluster variables and for any $f\in \text{Aut}\mathcal{A} $, $W$ is stable for $f$, i.e., $f(W)\subset W$, then $W$ is called a \textbf{stable monomial basis} for $\mathcal{A}$.
\end{defi}

\begin{defi}\label{defbasis}
	(i)\; Let $ \mathcal{A} $ be a cluster algebra and $ R $ its polynomial ring of frozen variables over $ \mathbb{Q} $.
	We choose a  set
$ C_{q}  $ such that elements in $C_{q} $ can be represented as Laurent polynomials of some cluster with positive coefficients.

Let $\mathcal{D} $ be a set of some subsets of $ C_{q} \cup \mathcal{X} $, such that for any $ U \in \mathcal{D} $, $ U \cap \mathcal{X} $ is contained in a cluster, where $\mathcal X$ is the set of all cluster variables.
	
	(ii)\; Besides, for $ z \in  C_{q} \cup \mathcal{X} $ we set
	\[ \langle k, z \rangle =\left\{ \begin{array} {ll}
		z^{k} & x \in \mathcal{X} \\
		P_{k}(z) & x \notin  \mathcal{X}
	\end{array} \right. \]
	where $ k \in \mathbb{Z}_{+} $ and $ P_{k}(z) \in R[z] $.
	
	(iii)\; Assume that for any $ f \in \text{\em Aut}\mathcal{A} $,
	\\
	(a)\; it holds $ f(\langle k, z \rangle)=\langle k, f(z) \rangle$ \;for any\; $ k \in \mathbb{N}, z \in U$;
	\\
	(b)\; $\mathcal{D}$ is stable for $ f $;
	\\
	(c)\;
	$ W=\bigcup_{U\in \mathcal{D}} \{ \prod_{x \in U'}\langle k_{x}, x \rangle | U' \subset U, k_{x} \in \mathbb{Z}_{+} \} $
	is a basis for $ \mathcal{A} $ and if $U_{1}\neq U_{2} $, then $\prod_{x \in U_{1}}\langle k_{x}, x \rangle \neq \prod_{x \in U_{2}}\langle k_{x}, x \rangle$.
	
	Then $ W $ is called a \textbf{$\mathcal{D}$-stable basis} for $\mathcal{A} $, the elements in $ C_{q} $ are called  \textbf{quasi cluster variables} of $\mathcal{A} $ and the element in $ \mathcal{D} $ is called a \textbf{maximal compatible set} of $ \mathcal{A} $.
	
\end{defi}

Suppose $W$ is a $\mathcal{D}$-stable basis for $ \mathcal{A} $ and $f \in\text{Aut}\mathcal{A} $. If $ \prod_{x \in U'}\langle k_{x}, x \rangle \in W$ for $ U' \subset U  \in \mathcal{D} $, then $ f(U) \in \mathcal{D}$ and $ f(U') \subset U $. Thus
$ f(\prod_{x \in U'}\langle k_{x}, x \rangle)=\prod_{f(x) \in f(U')}\langle k_{x}, f(x) \rangle \in W $ and hence $ W $ is also stable for any $ f \in \text{Aut}\mathcal{A} $. It means $ f $ is a bijection over $ W $.

Suppose $W$ is a stable monomial basis for  $ \mathcal{A} $. We say that $W$ has \textbf{unique expression} if  for $ v_{i}=x_{i1}^{k_{i1}}x_{i2}^{k_{i2}}...x_{is_{i}}^{k_{is_{i}}} \in W, i=1, 2$,  $v_{1}=v_{2}$ means $ s_{1}=s_{2} $ and there is a permutation $ \sigma \in S_{s_{1}} $ such that $ x_{1j}=x_{2\sigma(j)},$ for $ j=[1,s_{1}] $.

\begin{lem}\label{uniquequasicluster}
	Let $ \mathcal{A} $ be a cluster algebra with a stable monomial basis $W$. If the upper cluster algebra $ \mathcal{U} $ of $ \mathcal{A} $ is factorial, then $W$ has the unique expression.
\end{lem}
\begin{proof}
	For $ v_{1}, v_{2}\in W$, write $v_{i}=x_{i1}^{k_{i1}}x_{i2}^{k_{i2}}...x_{is_{i}}^{k_{is_{i}}},\; i=1, 2 $. Since $v_{1}=v_{2}$, we have
	\begin{equation}\label{mxmonomialseq}
		x_{11}^{k_{11}}x_{12}^{k_{12}}...x_{1s_{1}}^{k_{1s_{1}}} = x_{21}^{k_{21}}x_{22}^{k_{22}}...x_{2s_{2}}^{k_{2s_{2}}}.
	\end{equation}
	
	According to \cite{GLS13}, exchange variables are irreducible in $\mathcal{U}$, and invertible elements are Laurent monomials of frozen variables.
		If $ \mathcal{U} $ is factorial, $ x_{11}^{k_{11}}x_{12}^{k_{12}}...x_{1s_{1}}^{k_{1s_{1}}} = x_{21}^{k_{21}}x_{22}^{k_{22}}...x_{2s_{2}}^{k_{2s_{2}}}$ means
		the numbers of exchange variables in two sides are the same
		and there is a permutation $ \sigma' $ such that $ x_{1j}=p_{j}x_{2\sigma'(j)} $ for exchange variables $ x_{1j} $ and $ x_{2\sigma'(j)} $, where $ p_{j} $ are Laurent monomials of frozen variables.
	
		Let $ p_{j}=p_{j1}/p_{j2} $, where $ p_{j1} $ and $ p_{j2} $ are monomials of frozen variables. Then we have $p_{j2}x_{1j}=p_{j1}x_{2\sigma'(j)} $ in $ \mathcal{A} $. Note that both sides are extended cluster monomials.
		According to \cite{CaoLi20b}, extended cluster monomials are linearly independent. Thus $ x_{1j}= x_{2\sigma'(j)} $. Removing the exchange variables of both sides in (\ref{mxmonomialseq}), the monomials of the remaining frozen variables are equal.
		Because of algebraic independence of  frozen variables,
		$ s_{1}=s_{2} $ and there is a permutation $ \sigma \in S_{s_{1}} $ such that $ x_{1j}=x_{2\sigma(j)}, j\in [1,s_{1}] $.

\end{proof}

\begin{lem}\label{double}
	Let $ \mathcal{A}(\Sigma) $ be a cluster algebra with the extended exchange matrix of full rank, exchange variables $ x_{1},x_{2},x_{1}',x_{2}' \in \mathcal{A} $ and frozen variables $ y_{1}, ..., y_{s_{1}} $, $ y_{1}', ..., y_{s_{2}}' \in \mathcal{A} $. If $ y_{1}^{k_{1}}\cdots y_{s_{1}}^{k_{s_{1}}} x_{1}x_{2}= y_{1}'^{k_{1}}\cdots y_{s_{2}}'^{k_{s_{2}}} x_{1}'x_{2}' $ and the compatibility degree $d(x_1,x_2)\le 2$ for $ k_{1}, ...,k_{s_{1}}, k_{2}', ...,k_{s_{2}}' \in \mathbb{N} $, then $ x_{1}=x_{1}', x_{2}=x_{2}' $, or $ x_{1}=x_{2}', x_{2}=x_{1}' $, and $ s_{1} = s_{2}$ and there is a permutation $ \sigma \in S_{s_{1}} $ such that $ y_{i}= y_{\sigma(i)}', i\in [1,s_{1}] $.
\end{lem}

\begin{proof}
	Since $ x_{1},x_{2},x_{1}',x_{2}' $ are exchange variables of $\mathcal{A} $, $ x_{2}, x_{1}', x_{2}' $ can be express by a cluster $ \textbf{x}=\{ w_{1},w_{2},...,w_{m} \} $, where $ w_{1}=x_{1} $ as Laurent polynomials:
	\[ x_{2}=\frac{P_{2}(\textbf{x})}{D_{2}(\textbf{x})},\;\; x_{1}'=\frac{P_{1}'(\textbf{x})}{D_{1}'(\textbf{x})}, \;\; x_{2}'=\frac{P_{2}'(\textbf{x})}{D_{2}'(\textbf{x})}, \]
	where $ P_{2},P_{1}',P_{2}' $ are polynomials of all cluster variables $ \textbf{x} $ without monomials as their factors and $ D_{2},D_{1}',D_{2}' $ are Laurent monomials of $ \textbf{x} $. In this case, for $ i=1,2,...,n$, the compatibility degrees $ d(w_{i},x_{2}), d(w_{i},x_{1}')$, $ d(w_{i},x_{2}') $ are the degrees of $ D_{2},D_{1}',D_{2}' $ corresponding to $ w_{i} $ respectively. According to \cite{CaoLi20a}, we have $ d(w_{i},x_{2}), d(w_{i},x_{1}'),d(w_{i},x_{2}') \geq -1, i=1,2,...,n $, and
	\begin{equation}\label{dd}
		y_{1}^{k_{1}}\cdots y_{s_{1}}^{k_{s_{1}}} x_{1} \frac{P_{2}(\textbf{x})}{D_{2}(\textbf{x})} = y_{1}'^{k_{1}}\cdots y_{s_{2}}'^{k_{s_{2}}} \frac{P_{1}'(\textbf{x})}{D_{1}'(\textbf{x})} \frac{P_{2}'(\textbf{x})}{D_{2}'(\textbf{x})} \;\;
		\Longrightarrow \;\; y_{1}^{k_{1}}\cdots y_{s_{1}}^{k_{s_{1}}} x_{1}= y_{1}'^{k_{1}}\cdots y_{s_{2}}'^{k_{s_{2}}} \frac{D_{2}(\textbf{x})}{D_{1}'(\textbf{x})D_{2}'(\textbf{x})}\frac{P_{1}'(\textbf{x})P_{2}'(\textbf{x})}{P_{2}(\textbf{x})}.
	\end{equation}

	We now prove $ x_{1}' = x_{1} $ or $ x_{2}'= x_{1} $.
	
	Since $ P_{2},P_{1}',P_{2}' $ have not $ w_{1}=x_{1} $ as a factor, $ \frac{P_{1}'(\textbf{x})P_{2}'(\textbf{x})}{P_{2}(\textbf{x})} $ has not $ x_{1} $ as a factor. So by (\ref{dd}),
	$\frac{D_{2}(\textbf{x})}{D_{1}'(\textbf{x})D_{2}'(\textbf{x})} $ has $ x_{1} $ as a factor.
	
	Since $ d(x_1,x_2) \leq 2 $, it follows that $ D_{1}' $ or $ D_{2}' $ has not $ x_{1} $ as a factor.
	
	In the case if $ D_{1}' $ has not $ x_{1} $ as a factor. According to \cite{CaoLi20a}, there exists a cluster containing both $ x_{1} $ and $ x_{1}' $. We can assume that $ \textbf{x} $ also contain $ x_{1}' $, thus
		\[ y_{1}^{k_{1}}\cdots y_{s_{1}}^{k_{s_{1}}}
		x_{1} \frac{P_{2}(\textbf{x})}{D_{2}(\textbf{x})} = y_{1}'^{k_{1}}\cdots y_{s_{2}}'^{k_{s_{2}}} x_{1}' \frac{P_{2}'(\textbf{x})}{D_{2}'(\textbf{x})}.
		\]
	Because $ P_{2} $ and $ P_{1} $ do not have monomials as their factor, we have $ P_{2}=P_{2}' $.
	
	Assume $ \textbf{x} \in \Sigma $, where $ \Sigma=(\textbf{x}_{ex},\textbf{x}_{fr},\widetilde{B}_{m \times n}) $ is a seed of $ \mathcal{A} $, and suppose $ \mathcal{A}(\Sigma') $ be the cluster algebras with principle coefficients where initial seed $ \Sigma'=(\textbf{x}_{ex},\textbf{x}_{fr}',$ $\widetilde{B}'_{2n \times n} ) $, $ \textbf{x}_{ex}=\{w_{1},w_{2},...,w_{n} \} $, $ \textbf{x}_{fr}'= \{ w_{n+1}',w_{n+2}',...,w_{2n}' \}$.
	
	Since $\widetilde{B}$ has full rank, we have $ \widetilde{B}'=\Psi \widetilde{B}$ where $ \Psi=(\psi_{ij})_{2n \times m} $. Consider two cluster patterns
	\[
	(\Sigma(t))_{t \in \mathbb{T}_{n}} = (\textbf{x}_{ex}(t),\textbf{x}_{fr}(t),\widetilde{B}(t))_{t \in \mathbb{T}_{n}}, \;\;\;
	(\Sigma'(t))_{t \in \mathbb{T}_{n}} =(\textbf{x}_{ex}'(t),\textbf{x}_{fr}'(t),\widetilde{B}'(t))_{t \in \mathbb{T}_{n}}
	\]
	with initial seeds $ \Sigma(t_{0})=\Sigma $, $ \Sigma'(t_{0})=\Sigma' $ respectively. According to \cite{FWZ17}, if we define two morphisms $\gamma$ and $\varphi$ of algebras by setting
	\[ \gamma (w_{j})= \prod_{i=1}^{n}w_{i}^{\psi_{ij}}\prod_{i=n+1}^{2n}w_{i}'^{{\psi_{ij}}}, \;\;\;
	\varphi(w_{j})=\prod_{i=n+1}^{2n}w_{i}'^{\psi_{ij}}, \;\;\; j=1,...,m,\]
	then we have
	$ x_{i;t}'=\frac{\gamma(x_{i;t})}{\varphi(x_{i,t})}. $
	
	Thus for $ x_{2}, x_{2}' $, we can set
	\[ z_{2}= \frac{\gamma(P_{2})/\gamma(D_{2})}{\varphi(P_{2})/\varphi(D_{2})}, z_{2}'= \frac{\gamma(P_{2}')/\gamma(D_2')}{\varphi(P_{2}')/\varphi(D_{2}')}, \]
	which are exchange variables in $ \mathcal{A}(\Sigma') $. Since $ \psi_{ij}=0,i \neq j; \,\phi_{ii}=1 $, for $ i,j \le n $, then we have $ \gamma (w_{j}) =w_{j}\varphi(w_{j}), j=1,...,m $, and hence $ \varphi(D_{2})/\gamma(D_{2}) =D_{2}$, $ \varphi(D_{2}')/\gamma(D_{2}') =D_{2}' $. Thus the $ F $-polynomial of $ z_{2}$ is equal to
	\[ {\frac{\gamma(P_{2})/\gamma(D_{2})}{\varphi(P_{2})/\varphi(D_{2})}}\Big|_{w_{1}=\cdots =w_{n}=1} =\frac{\gamma(P_{2})}{\varphi(P_{2})}\frac{\varphi(D_{2})}{\gamma(D_{2})}\Big|_{w_{1}=\cdots =w_{n}=1}=
	{\frac{\gamma(P_{2})}{\varphi(P_{2})}}\Big|_{w_{1}=\cdots =w_{n}=1}, \]
	and the $ F $-polynomial of $ z'_{2}$ is equal to $ \gamma(P_{2}') / \varphi(P_{2}') |_{w_{1}=\cdots =w_{n}=1} $. They are equal since $P_2=P'_2$, which means $ z_{2}= z_{2}' $. On the other hand, $ \widetilde{B}=\Psi' \widetilde{B}'$, so according to \cite{FWZ17}, we have $ x_{2}=x_{2}' $ and hence $ x_{1}=x_{1}' $.
	
	In the case if $ D_{2}' $ does not have $ x_{1} $ as a factor, we can similarly prove $ x_{2}=x_{1}' $ and $ x_{1}=x_{2}' $.
	
	Considering the algebraically independence of frozen variables, we have $ s_{1} = s_{2}$ and there is a permutation $ \sigma \in S_{s_{1}} $ such that $ y_{i}= y_{\sigma(i)}', i\in [1,s_{1}] $.

\end{proof}

Thus we have the following corollary.
\begin{cor}\label{2unique}
	Let $ \mathcal{A}$ be a cluster algebra with a stable monomial basis $W=\{w_{\lambda}\}_{\lambda \in \Lambda}$ and let the extended exchange matrices of $\mathcal A$ be of full rank. If the degrees of elements in $W$ except extended cluster monomials are $2$ and the compatibility degree of each two exchange variables in the expression of $w_{\lambda}$ is not larger than $2$, then $W$ has the unique expression.
\end{cor}
\begin{proof} Due to the condition, any pair $v_{i}, i=1, 2$ in $W$ except extended cluster monomials can be expressed as $ v_{i}=p_{i}x_{i1}x_{i2}$, where $ p_{i} $ are a monomial of frozen variables and $ x_{1}, x_{2} $ are cluster variables such that the compatibility degree $ d(x_{i1},x_{i2}) \le 2 $. If $v_{1}= v_{2}$, we have $ p_{1}x_{11}x_{12}=p_{2}x_{21}x_{22} $. The corollary follows immediately from Lemma \ref{double}.

\end{proof}

Note that  a standard monomial basis is not necessarily a  $\mathcal{D}$-stable basis and vice versa.

According to \cite{CK08}, all extended cluster monomials form a basis if the cluster algebra is of finite type, Thus the set of all extended cluster monomials from the cluster algebra of finite type is a stable monomial basis as well as a $\mathcal{D}$-stable basis. 	

We have also some examples from cluster algebras of non-finite type. For example, the bases $ \mathcal{B}^{\circ} $, $ \mathcal{B} $ and $ \mathcal{B}' $ given in \cite{MSW13, Thu14} without coefficient are $\mathcal{D}$-stable bases, which is introduced and proved in the next section.

If a cluster algebra $\mathcal{A}$ is acyclic, i.e., there is an acyclic exchange matrix in a seed of $\mathcal{A}$, a standard monomial basis for $\mathcal{A}$ was introduced in \cite{BFZ05}, which is a stable monomial basis with the unique expression.

If $\mathcal A$ has a $\mathcal{D}$-stable basis, then we give a discussion on the condition for $ \mathcal{A}(\Sigma_{1})\subseteq\mathcal{A}(\Sigma)$ to be a Galois extension.

\begin{thm}\label{crit}
		Let $\mathcal{A}=\mathcal{A}(\Sigma) $ be a cluster algebra and $ H $ be a subgroup of the cluster automorphism group of $\mathcal{A}(\Sigma)$. Suppose $ \mathcal{A}(\Sigma_{1}) $ is maximal in $ \mathcal{A}^{H} $ as a cluster subalgebra of $ \mathcal{A}$.

	(i)\; If $ \mathcal{A}^{H} = \mathcal{A}(\Sigma_{1}) $, then $ |Hz| = + \infty $ for any $ z \in \mathcal{A} \backslash \mathcal{A}(\Sigma_{1}) $ which can be represented as Laurent polynomials of some cluster with positive coefficients.
	
	(ii)\; Conversely, $ \mathcal{A}^{H} = \mathcal{A}(\Sigma_{1}) $ holds if either (a)   $\mathcal A$ has a stable monomial basis with the unique expression  and $ |Hz| = + \infty $ for any cluster variable $ z \in \mathcal{A} \backslash \mathcal{A}(\Sigma_{1}) $ or
	(b)   $\mathcal A$ has a $\mathcal{D}$-stable basis  $ |Hz| = + \infty $ for any cluster variable or quasi cluster variable $ z \in \mathcal{A} \backslash \mathcal{A}(\Sigma_{1}) $.
\end{thm}
\begin{proof}
	(i)\; It is equivalent to prove that if there exists a cluster variable $ x \in \mathcal{A} \backslash \mathcal{A}(\Sigma_{1}) $ such that $ |Hx| < + \infty $, then $ \mathcal{A}^{H} \neq \mathcal{A}(\Sigma_{1}) $.
	
	According to Theorem \ref{subseed}, we can assume $ \Sigma_{1} $ is a mixing-type sub-seed of $ \Sigma $. If $ \Sigma_{1} $ contains all cluster variables in $ \Sigma $, then $ H $ is the trivial group and the sufficiency is obvious.
	
	Thus we can assume $ \textbf{x}'\neq \textbf{x} $, where $ \textbf{x} $ (resp. \textbf{x}') is the cluster of $ \Sigma $ (resp. $ \Sigma_{1} $). Suppose $ z \in \mathcal{A} \backslash \mathcal{A}(\Sigma_{1}) $ and $ |Hz| =\{z_{i},i\in[1,t] \}, z_{1}=z $. The Laurent expression of $ z_{i} $ with respect to $ \textbf{x} $ are
	\[ z_{i}=\frac{P_{i}(\textbf{x})}{D_{i}(\textbf{x})}, i \in [1,t]\]
	where $ P_{i}(\textbf{x}) $ are polynomials of $ \textbf{x} $ with positive coefficients and $ D_{i}(\textbf{x}) $ are extended cluster monomials of $ \textbf{x} $. For any monomial $M$ of $\textbf{x}'$, if $z_{1}M \in \mathcal{A}(\Sigma_{1})$, then $f(z_{1})=z_{1}$, $f\in H$. Thus $z_{1} \in \mathcal{A}^{H}\backslash \mathcal{A}(\Sigma_{1})$ and this completes the proof. Otherwise, we can replace $z_{1}$ with $z_{1}M$ for some monomial $M$ such that $\frac{P_{i}(\textbf{x})}{D_{i}(\textbf{x})}, i\in [1,t]$ are polynomials with respect to $\textbf{x}'$.
		
	Let $ Z = \sum_{i=1}^{t}z_{i} $, then $ Z \in \mathcal{A}^{H} $. Assume $ Z \in \mathcal{A}(\Sigma_{1}) $, then $ z= L(\textbf{x}') $, where $ L(\textbf{x}') $ is a Laurent polynomial of $ \textbf{x}' $. And hence we have
	\begin{equation}\label{exelemet}
		L(\textbf{x}')= \sum_{i}^{t}\frac{P_{i}(\textbf{x})}{D_{i}(\textbf{x})}.
	\end{equation}
	Since $ z_{1} \in \mathcal{A} \backslash \mathcal{A}(\Sigma_{1}) $, there exists $ x\in \textbf{x} \backslash \textbf{x}' $ such that $ x $ appears in $ \frac{P_{1}(\textbf{x})}{D_{1}(\textbf{x})} $. Since coefficients of $P_{i}(\textbf{x}), i\in [1, t]$ are positive, $ x $ also appears in the right of (\ref{exelemet}).
	But $ x $ does not appear on left of (\ref{exelemet}), which means the equality does not hold. Thus $ z \notin \mathcal{A}(\Sigma_{1}) $ and hence $ \mathcal{A}^{H} \neq \mathcal{A}(\Sigma_{1}) $.
	
	(ii)\; It only needs to prove under the assumption that $ \mathcal{A} $ has a stable monomial basis with the unique expression (resp. $\mathcal{D}$-stable basis) and $ \mathcal{A}^{H} \neq \mathcal{A}(\Sigma_{1}) $, it following that there exists a cluster variable (resp. cluster variable or quasi cluster variable) $ z \in \mathcal{A} \backslash \mathcal{A}(\Sigma_{1}) $ such that $ |Hz| < + \infty $.
	
	Suppose $ \mathcal{A}^{H} \neq \mathcal{A}(\Sigma_{1}) $, then there exists $ v \in \mathcal{A} \backslash \mathcal{A}(\Sigma_{1}) $, such that $ f(v)=v $, for any $ f \in H $.
	Since $\mathcal A$ has a stable monomial basis  (resp. $\mathcal{D}$-stable basis) $W$, then $z$ can be represented by $ v=k_{1}v_{1}+k_{2}v_{2}+\cdots +k_{s}v_{s}, 1,\cdots,s\in I$ and $ k_{1}, ..., k_{s} \in \mathbb{Q} $. Thus
	\[ f(z)=k_{1}f(v_{1})+k_{2}f(v_{2})+ \cdots +k_{s}f(v_{s})= k_{1}v_{1}+k_{2}v_{2}+ \cdots +k_{s}v_{s} . \]
	Because $ W$ is stable for $f$, the equality holds if and only if $ f $ is bijective over $ \{v_{1},v_{2},...,v_{s} \} $.
	
	Since $ v \in \mathcal{A} \backslash \mathcal{A}(\Sigma_{1}) $, there exists $ v_{k} $ such that $ v_{k} \in \mathcal{A} \backslash \mathcal{A}(\Sigma_{1}) $.
	
	(a) \, In the case $W$ is a stable monomial basis, denote by $X$ the set of all cluster variables appearing in expressions of $v_{i}, i \in [1, s]$. Then $f$ restricted to $X$ is bijective. Since $ v_{k} \in \mathcal{A} \backslash \mathcal{A}(\Sigma_{1}) $, there is $z \in X$ such that $z \in \mathcal{A} \backslash \mathcal{A}(\Sigma_{1})$. Because of the arbitrariness of $ f $, we have $|Hz| < + \infty $.
	
	(b) \, In the case $W$ is a $\mathcal{D}$-stable basis, write $ v_{j} = \prod_{x \in U_{j}}\langle k_{x}, x \rangle, j=[1,s], k_{x} \in \mathbb{Z}_{+} $. Denote by $X_{1}$ and $X_{2}$ all cluster variables and quasi cluster variables in $\bigcup_{j=1}^{s} U_{j} $, respectively. Since the set of cluster variables is stable for $f$ and according to Definition \ref{defbasis} (iii), $f$ restricted to $X_{1}$ and $X_{2}$ are bijective. Since $ v_{k} \in \mathcal{A} \backslash \mathcal{A}(\Sigma_{1}) $, there is a cluster variable $z \in X_{1}$ or quasi cluster variables $z \in X_{2}$ such that $z \in \mathcal{A} \backslash \mathcal{A}(\Sigma_{1})$. Because of the arbitrariness of $ f $, we have $|Hz| < + \infty $.
	
\end{proof}

According to Lemma \ref{uniquequasicluster} and Corollary \ref{2unique}, we have
\begin{cor}\label{corcrit}
	Let $\mathcal{A}=\mathcal{A}(\Sigma) $ be a cluster algebra with stable monomial basis $W=\{w_{\lambda}\}_{\lambda \in \Lambda}$ and a subgroup $ H \le \text{Aut} \mathcal{A}$ such that there is  a  maximal cluster subalgebra  $ \mathcal{A}(\Sigma_{1}) $ in $ \mathcal{A}^{H} $.
	Suppose $ |Hx| = + \infty $ for any cluster variable $ x \in \mathcal{A} \backslash \mathcal{A}(\Sigma_{1}) $ and
	\\
	either (i) the upper cluster algebra $ \mathcal{U} $ of $ \mathcal{A} $ is factorial
	\\
	or (ii) the extended exchange matrices of $\mathcal A$ are of full rank and the degrees of elements in $W$ except extended cluster monomials are $2$ and the compatibility degree of each two exchange variables in the expression of $w_{\lambda}$ is not larger than $2$.
	
	Then $ \mathcal{A}^{H} = \mathcal{A}(\Sigma_{1}) $.
\end{cor}

According to Theorem \ref{crit}, we have the following remark.

	\begin{rmk}
		For any skew-symmetrizable cluster algebra $\mathcal A$, if $ \text{\em Aut}\mathcal{A} $ is finite, then $\text{\em Aut}\mathcal{A}x < + \infty $ for any cluster variable x.
		According to Theorem \ref{crit} (i), any cluster subalgebra $ \mathcal{A}(\Sigma') \subsetneq \mathcal{A} $ is not Galois extension subalgebra.
		\\
		If $\mathcal A$ is the cluster algebra with principal coefficients, according to proposition \ref{exam}, $ \text{\em Aut}\mathcal{A} $ is finite and hence in this case any cluster subalgebra $ \mathcal{A}(\Sigma') \subsetneq \mathcal{A} $ is not Galois extension subalgebra.
	\end{rmk}

\section{Examples of $\mathcal{D}$-stable bases and Galois theory of cluster algebras from surfaces}

In this section, we give some examples of $\mathcal{D}$-stable bases for the cluster algebras from surfaces and discuss the Galois correspondence and Galois-like extensions for this cluster algebras.

Suppose $ \mathcal{A} $ is a cluster algebra with coefficients and $ \mathcal{A}' $ is the cluster algebra obtained from $ \mathcal{A}$ by specializing all frozen variables to 1, then $\mathcal A'$ is a cluster algebra without coefficients. Note that cluster automorphisms of $ \mathcal{A} $ are also automorphisms of $ \mathcal{A}'$ by specializing all frozen variables to 1, but in general, cluster automorphisms of $ \mathcal{A}' $ are not necessarily that of $ \mathcal{A} $.
So, in order to let the cluster automorphism group be as large as possible, we always suppose the corresponding cluster algebras of marked surfaces are without coefficient in this section if there is no additional explanation.
\subsection{Bangle, bracelet, and band bases as $\mathcal{D}$-stable bases }\quad

Let $ (S, M) $ be a marked surface, $ h $ be a homeomorphism of $ (S, M) $ such that $ h(M)=M $ and $ R $ is a subset of the set consisting of all punctures in $ (S, M) $, then we can obtain a cluster automorphisms of $ \mathcal{A}(S, M) $ corresponding to homeomorphism of $ (S, M) $ induced by $ h $ and $ R $:
\[ \psi_{h} \in \text{Aut}\mathcal{A}(S,M) \quad \text{defined by } \psi_{h}(x_{\gamma}):=x_{h(\gamma)} , \]
where $ \gamma \in A_{\bowtie}(S,M) $ the set consisting of all tagged arcs of $ (S,M) $ and $ x_{\gamma} $ is its counterpart cluster variable in $ \mathcal{A}(S,M) $. Besides, we set
\[ \psi_{R} \in \text{Aut}\mathcal{A}(S,M) \quad \text{defined by } \psi_{R}(x_{\gamma}):=x_{\gamma^{R}}, \]
where $ \gamma^{R} $ denote the tagged arc obtained from $ \gamma $ by changing the taggings of it at those of its endpoints that belong to $ R $. For $ h $ and $ R $ as above, denote by
$ \psi_{h,R}:=\psi_{R}\psi_{h}. $

According to \cite{Gu11,BS15,BQiu15,BY18}, we have the following theorem:

\begin{thm}\label{mcg}\cite{Gu11,BS15,BQiu15,BY18}
	Let $ (S, M) $ be a marked surface different from any one of
	(1) the 4-punctured sphere,
	(2) the once-punctured 4-gon,
	(3) the twice-punctured digon. Then the group of cluster automorphisms of the corresponding cluster algebra $\mathcal A(S, M)$ is given as follows:
	
	(a) if $ (S, M) $ is a once-punctured closed surface, then
	\[\mathcal{MCG}^{\pm}(S,M) \cong \text{\em Aut}\mathcal{A}(S,M), \; \mathcal{MCG}(S,M) \cong \text{\em Aut}^{+}\mathcal{A}(S,M) \;\;\; via\;\; \;
	h \mapsto \psi_{h},
	\]
	
	(b) if $ (S, M) $ is not a once-punctured closed surface, then
	\[\mathcal{MCG}^{\pm}_{\bowtie}(S,M) \cong \text{\em Aut}\mathcal{A}(S,M), \; \mathcal{MCG}_{\bowtie}(S,M) \cong \text{\em Aut}^{+}\mathcal{A}(S,M) \;\;\; via\;\;\;
	(h, R) \mapsto \psi_{h,R},
	\]
	where $ h $ is a representative element of a class in $ \mathcal{MCG}^{\pm}(S,M)$, $ R $ is a subset of sets of all punctures in $ (S,M)$.
\end{thm}

This theorem builds up a connection between mapping class groups and cluster automorphism groups for the surfaces satisfying the conditions (1)--(3).
Both the results in the above existing literature and that obtained below show that the kind of surfaces satisfying these three conditions is very important and useful.
For convenience, we give the definition for such surfaces as follows:
\begin{defi}\label{def-feasiblesurface}
	A marked surface $ (S,M) $ is called a \textbf{feasible surface} if $ (S,M) $ is different from any one of
	(1) the 4-punctured sphere, (2) the once-punctured 4-gon, and
	(3) the twice-punctured digon.
\end{defi}

\begin{figure}[htbp]
	
	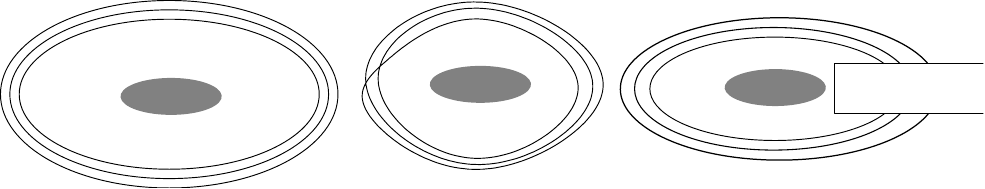
	\caption{A bangle $\text{Bang}_3 \zeta$(left), a bracelet $\text{Brac}_3 \zeta$(middle) and a band $ \text{Band}_{3} \zeta $ (right).}
	\label{f3}
\end{figure}

Next, we introduce some material about bases for cluster algebras from surfaces.

Suppose $ \zeta $ is an essential loop i.e., a loop in the interior of a surface $S$   which is neither a tagged arc nor contractible nor contractible onto a single puncture. The \textbf{bangle} $\text{Bang}_k \zeta$ is defined to be the union of $k$ copies of $\zeta$. The \textbf{bracelet} $\text{Brac}_k \zeta$ is defined to be the closed loop obtained by concatenating $\zeta$ exactly $k$ times. Suppose $ I \subset \zeta $ is a short interval, the \textbf{band} $\text{Band}_k \zeta$ is defined to be $k$ copies of $\zeta \backslash I $ with the ends connected by averaging all ways of pairing the endpoints on the two sides under the formal addition, see Figure \ref{f3}.

According to \cite{MSW13, Thu14}, essential loops, bangles, bracelets, and bands correspond to elements of the upper cluster algebra of $ \mathcal{A}(S, M) $, respectively.

The Chebyshev polynomials of the first kind are polynomials $ T_{k}(z) $ defined recursively as follows:
\[ T_{0}(z)=2, \, \, T_{1}(z)=z,\;\; \text{and }\;\; T_{n+1}(z)=zT_{n}(z)-T_{n-1}(z)\;\;\; \forall n\in \mathbb{N}, \]
and the Chebyshev polynomials of the second kind are polynomials $ U_{k}(z) $ defined recursively as follows:
\[ U_{0}(z)=1, \, \, U_{1}(z)=z, \;\; \text{and }\;\;
U_{n+1}(z)=zU_{n}(z)-U_{n-1}(z) \;\;\; \forall n\in \mathbb N. \]
According to \cite{MSW13,Thu14}, we have $ x_{\text{Bang}_{k}\zeta}= x_{\zeta}^{k}$, $ x_{\text{Brac}_{k}\zeta}= T_{k}(x_{\zeta})$ and $ x_{\text{Band}_{k}\zeta}= U_{k}(x_{\zeta}) $,
where $ x_{\zeta} $ is the element corresponding to the essential loop $\zeta $.

Suppose $\gamma, \gamma_1$ and $\gamma_2$ are generalized arcs or closed loops in $ (S,M) $ satisfying: \\
(1)\, $\gamma_1$ and $\gamma_2$ intersect at a point $x$, \\
(2)\, $\gamma$ is self-intersecting at a point $x$. \\
Then we let $C$ be the multicurve $\{\gamma_1,\gamma_2\}$ or $\{\gamma\}$ depending on which of the above two cases we suppose. We define the \textbf{smoothing} of $C$ at the point $x$ to be the pair of multicurves $C_+ = \{\alpha_1,\alpha_2\}$ (resp. $\{\alpha\}$) and $C_- = \{\beta_1,\beta_2\}$ (resp. $\{\beta\}$).

Here, the multicurve $C_+$ (resp. $C_-$) is the same as $C$ except for the local change via replacing the crossing {\Large $\times$} with the pair of segments {\LARGE $~_\cap^{\cup}$}
(resp. {\large $\supset \subset$}).

In a local region, we can choose an arc $\alpha$ for any essential loop $\zeta$ such that the smoothing of $\zeta$ and $\alpha$ is given as in Figure \ref{fig-smooth}.  Thus the corresponding element $x_{\zeta}$ can be represented as a Laurent polynomial of some cluster with positive coefficients.

\begin{figure}[htbp]
	\centering
	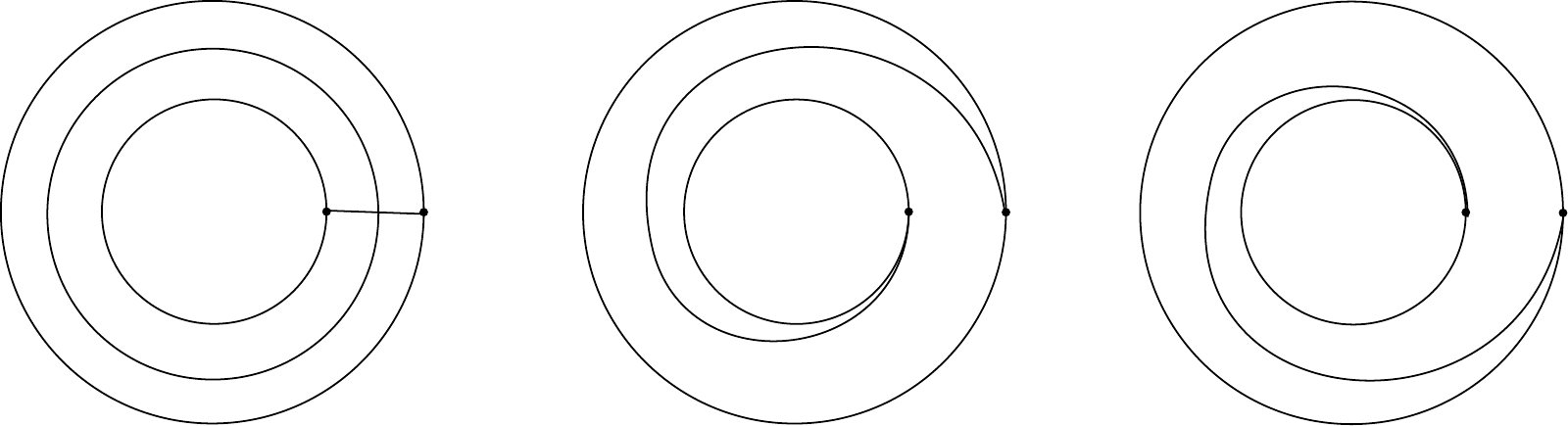
	\caption{The smoothing of $\zeta$ and $\alpha$}
	\label{fig-smooth}
\end{figure}

As known well, a {\bf multiset} means a gathering of some elements allowing the same elements included more tha once.

\begin{defi}
	(i) \, A finite multiset $ C^{\circ} $ of arcs and bangles is said to be \textbf{$ \mathcal{C}^{\circ} $-compatible} if no two elements of $ C^{\circ} $ intersect with each other and there is at most one bangle for each essential loop in $ C^{\circ} $.
	
	(ii) \, A finite multiset $ C $ of arcs and bracelet is said to be \textbf{$ \mathcal{C} $-compatible} if no two elements of $ C $ intersect with each other, except for the self-intersection of a bracelet, and there is at most one bracelet for each essential loop in $ C $.
	
	(iii) \, A finite multiset $ C' $ of arcs and band is said to be \textbf{$ \mathcal{C}' $-compatible} if no two elements of $ C' $ intersect with each other, except for the self-intersection of a band, and there is at most one band for each essential loop in $ C' $.
	
\end{defi}

Suppose $\mathcal{C}^{\circ}$ (resp. $ \mathcal{C} $ and $ \mathcal{C}' $) is the set consisting of all $\mathcal{C}^{\circ}$-compatible (resp. $\mathcal{C}$-compatible  and $\mathcal{C}'$-compatible) multisets.
For any surface $ (S, M) $ with at least two marked points and without punctures, suppose $ \mathcal{A}(S, M) $ is the cluster algebra from $ (S, M) $, then the {\bf Bangle basis} $ \mathcal{B}^{\circ} $, the {\bf Bracelet basis} $ \mathcal{B} $ and the {\bf Band basis} $ \mathcal{B}' $ of $ \mathcal{A}(S, M) $ are defined in \cite{MSW13} and \cite{Thu14} respectively as follows:
\begin{equation}\label{circ}
	\mathcal{B}^{\circ}= \{ \prod_{c\in C^{\circ}}x_{c} |\; C^{\circ} \in \mathcal{C}^{\circ} \};
\;\;\;
\mathcal{B}=\{ \prod_{c\in C}x_{c} |\; C \in  \mathcal{C} \}; \;\;\;
\mathcal{B}'= \{ \prod_{c\in C'}x_{c} |\; C'\in \mathcal{C}'  \}.
\end{equation}

Set the elements corresponding to essential loops as quasi cluster variables and the maximal sets of $ \mathcal{C}^{\circ} $-compatible (resp. $ \mathcal{C} $-compatible, $ \mathcal{C}' $-compatible) elements as elements of maximal compatible sets, then the set of maximal compatible sets $\mathcal{D}$ is stable for any  $f \in \text{Aut}\mathcal{A}(S,M)$.

Thus the Bangle basis $ \mathcal{B}^{\circ}$, the Bracelet basis $\mathcal{B} $ and the Band basis $\mathcal{B}' $ satisfying the Definition \ref{defbasis} and hence they are $\mathcal{D}$-stable bases for $ \mathcal{A}(S,M) $.

\subsection{Reverse Galois maps and the Galois inverse problem for cluster algebras from feasible surfaces} \label{surface-conj-subsection}   \quad

Now we prove Conjecture \ref{ssq} and \ref{srk} for cluster algebras from surfaces with certain conditions.

Recall that a tagged arc $ \gamma $ with the natural orientation is also denoted by $ \gamma $ and the tagged arc with the reverse orientation is denoted by $ \gamma^{-1} $.  Denote by $ \rho' $ the isomorphism from the cluster automorphism group to the mapping class group in Theorem \ref{mcg}. Define a map $\tau$ which assigns each mapping class to a representative in this mapping class. Let $ \rho= \tau \rho' $.

\begin{lem}\label{tagstable}
	Let $ (S,M) $ be a feasible surface, $ \mathcal{A}(S,M) $ its cluster algebra without coefficient, and $ H \le \text{\em Aut}\mathcal{A}(S,M) $. If $ x_{\gamma} \in \mathcal{A}(S,M)^{H} $ for some oriented tagged arc $ \gamma $ in $ (S,M) $ and $ \rho(f)(\gamma)=\gamma$ for any $ f\in H $,  then $ x_{\gamma^{p}} \in \mathcal{A}(S,M)^{H} $, where $ p $ is a puncture as well as an endpoint of $ \gamma $ and $ \gamma^{p} $ are the tagged arc obtained from $ \gamma $ by changed the tagging at $ p $.
\end{lem}
\begin{proof}
	For each $ f \in H $, we also consider $ f $ as an element in the mapping class group.
	Since $ \rho(f)(\gamma)=\gamma$, the endpoints of $ \gamma $ is unchanged under $ f $. Let $ \psi_{p} \in \mathcal{A}(S,M)^{H} $ such that $ \psi_{p}(x_{\alpha})=x_{\alpha^{p}} $ for any tagged arc $ \alpha $. Then $ f $ and $ \psi_{p} $ commute with each other if we restrict them to $ \gamma $. Thus $ f(x_{\gamma^{p}})=f(\psi_{p}(x_{\gamma}))=\psi_{p}(f(x_{\gamma}))=x_{\gamma^{p}} $. Because of the arbitrariness of $ f $, we have $ x_{\gamma^{p}} \in \mathcal{A}(S,M)^{H} $.
\end{proof}

\begin{rmk}\label{extodirect}
	Suppose $ x $ is an exchange variable in $ \mathcal{A}(S,M)^{H} $, i.e.,  there is a mixing-type sub-seed $ \Sigma' $ in $ \mathcal{A}(S,M)^{H} $ of $ \Sigma $ such that $ x $ is exchange variable in $ \Sigma' $.
	In this case, for each variable $ y \in \Sigma $, if $ b_{xy}\neq 0 $, then $ y \in \Sigma' $.

   Suppose $ \gamma, \gamma_{1} $ correspond to $ x $ and $ y $ respectively, such that $ b_{xy}\neq 0 $. Since $ \rho(f)$ maps $ \gamma, \gamma_{1} $ to themselves for each $f \in H$, $ b_{f(x)f(y)}=b_{xy} $, thus $ f $ is a direct cluster automorphism. Hence $ H \le \text{\em Aut}^{+}\mathcal{A}(S,M) $.
	
	Besides, let $ \gamma $ be an oriented arc corresponding to an exchange variable $ x_{\gamma} $ in $ \mathcal{A}(S,M)^{H} $, then for any $ f \in H $, $ \rho(f) $ maps $ \gamma $ to itself preserving the orientation.
\end{rmk}

\begin{lem}\label{asymnosymlem}
	 Let $ \mathcal{A}(S,M) $ be a cluster algebra without coefficient from a feasible surface $(S,M)$, and for any $ f\in \text{\em Aut}^{+}\mathcal{A}(S,M) $, let $ x_{\gamma} $ belong to a mixing-type sub-seed $ \Sigma' $ in $\mathcal{A}(S,M)^{f} $ such that $ \rho(f)(\gamma)=\gamma^{-1} $. Then there is no oriented tagged arc $ \delta $ corresponding to cluster variables $ x_{\delta} \in \Sigma' $ such that $ \rho(f)(\delta)=\delta $.
	
\end{lem}
\begin{proof}
	If $ \rho(f)(\gamma)=\gamma^{-1}$, then the taggings at two endpoints $ p, q $ of $ \gamma $ are the same and we also have $ \rho(f)(\gamma^{pq})=(\gamma^{pq})^{-1} $. Combining with Lemma \ref{tagstable}, This conclusion holds for tagged arcs corresponding to cluster variables in $ \Sigma' $ if and only if it holds for the tagged arcs obtained from these arcs corresponding to $ \Sigma' $ by changing the tagging at $ p,q $ at the same time or at other punctures. Thus we can assume tagged arcs corresponding to cluster variables in $ \Sigma' $ are contained in a triangulation $ T $ satisfying all punctures are tagged plain except for some punctures tagged in different ways in $ T $.
	
	Now we assume there is an oriented tagged arc $ \delta $ such that $ \rho(f)(\delta)=\delta $. In this case, $ \gamma $ and $ \delta $ do not tag notched at their endpoints, since $ \rho(f)(\gamma)=\gamma^{-1}$ and $ \rho(f)(\delta)=\delta $. We can do flip at each tagged arc that tags notched at its endpoint to obtain a triangulation whose tagged arcs tag plain. Thus we can assume all arcs in $ T $ tagging plain.

	Suppose $ \gamma $ is adjacent to $ \alpha_{1} $, $ \alpha_{i} $ is adjacent to $ \alpha_{i+1} $ for $ i \in [1, s-1] $ and $ \alpha_{s} $ is adjacent to $ \delta $ in $ T $, $ \rho(f)(\alpha_{i})=\beta_{i} $ for $ i\in [1,s] $. Then $ \gamma $ is adjacent to $ \beta_{1} $, $ \beta_{i} $ is adjacent to $ \beta_{i+1} $ for $ i \in [1, s-1] $ and $ \beta_{s} $ is adjacent to $ \delta $ (see Figure \ref{sqfig}).
	
	Since $ f \in \text{Aut}^{+}\mathcal{A}(S,M) $ and $ \rho(f)(\gamma)=\gamma^{-1} $, $ \rho(f) $ is the $ 180^{\circ} $ rotation of the local region in Figure \ref{sqfig}. Thus $\rho(f) $ reverses the orientation of $ \delta $ in Figure \ref{sqfig}, which is contrary to that the surface $ S $ is oriented.
\end{proof}

\begin{figure}[htbp]\label{sqfig}
	\centering
	\includegraphics[scale =0.25]{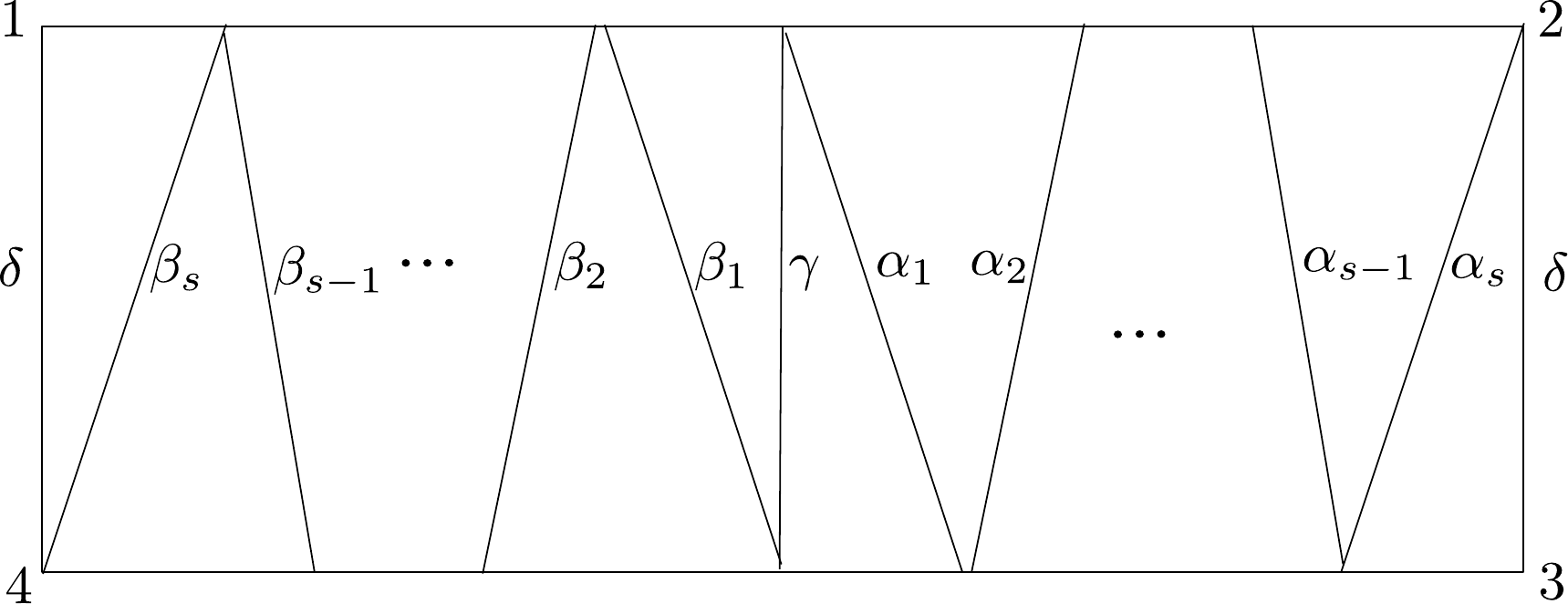}
	\caption{The local region from $ \gamma $ to $ \delta $}
\end{figure}

\begin{lem} \label{symnoasymlem}
	
		Let $ (S, M) $ be a feasible surface, $ \mathcal{A}(S, M) $ its cluster algebra without coefficient,
		and there is an oriented tagged arc $ \delta $ such that $ \rho(f)(\delta)=\delta $ for some $ f\in \text{\em Aut}^{+}\mathcal{A}(S, M) $.
		Then there is no oriented tagged arc $ \gamma $ such that $ \rho(f)(\gamma)=\gamma^{-1} $.
\end{lem}
\begin{proof}
	Similarly, we can assume all arcs tagged plain.
	 Asume there is an oriented arc $ \gamma $ such that $ \rho(f)(\gamma)=\gamma^{-1} $.
	If $ \gamma $ intersects with $ \delta $, then $ \rho(f) $ reverses the orientation of local region contains $ \gamma $ and $ \delta $, which is contrary to $ f\in \text{Aut}^{+}\mathcal{A}(S,M) $. If there is no intersection between $ \gamma $ and $ \delta $, then there is a mixing-type sub-seed $ \Sigma' $ in $ \mathcal{A}(S, M)^{f} $ contains the cluster variables corresponding to $ \gamma $ and $ \delta $, which is contrary to Lemma \ref{asymnosymlem}.
	
\end{proof}

  $ f \in \text{Aut}^{+}\mathcal{A}(S,M)  $ is said to be \textbf{orientation-preserved} if  there is an oriented arc $ \gamma $ such that $ \rho(f)(\gamma)=\gamma $.

 A \textbf{subsurface} $C$ of a marked surface $(S, M)$ is a closed subset of $S$ which  is a two-dimensional submanifold and whose boundary components are a loop in $(S, M)$ which is neither contractible onto an arc, a boundary arc or a single puncture  nor contractible except the case that one of connected components of $C$ is the digon.
 Note that  here a loop $\zeta$ is defined to {\bf be contractible onto} $X\subset S$ (resp. {\bf contractible}) if  there is a isotopy from the identity of $(S,M)$ that transforms $\zeta$ into $X$ which fixes $X$ (resp. any point in $S \backslash (\partial S \cup M)$).

 We suppose that for a marked surface $(S,M)$, an open disk, open once-punctured disk or open annulus mentioned in the sequel, which is a connected component of $S\backslash C $, has no marked point on its boundary.

Suppose $ C $ is a subsurface of $ (S,M) $, we consider  maximal number of compatible tagged arcs of $ (S,M) $ which lie in $ (C,C\cap M) $.

	We call a set of compatible tagged arcs lying in $(C, C\cap M) $ which is maximal a \textbf{maximal tagged arcs set} of $(C, C\cap M) $.

\begin{lem} \label{subnarclem}
		Let $ (S, M) $ be a marked surface, $ C $ a connected subsurface of $ (S, M) $. Then the cardinalities $n$ of maximal tagged arcs sets of $ (S, M) $ are constant.
		
 Concretely, we can obtain the valuation of the cardinalities $\mathfrak{N}$ respectively in the following various cases: \\
		\indent (1) the cardinality $ \mathfrak{N}=0 $ in the case  $ (C, C\cap M)  $ is a subsurface with no marked point or the once-punctured disk; \\
		\indent (2) the cardinality $ \mathfrak{N}=2$ in the case  $ (C, C\cap M)  $ is the once-punctured disk with 1 marked point on the boundaries;  \\
		\indent (3) the cardinality $ \mathfrak{N}=n_{b}$ in the case  $ (C, C\cap M)  $ is the annulus or the digon; \\
		\indent (4) the cardinality $ \mathfrak{N}=6g+3a+3p+c-6-t-s-d+n_{b}$ in other cases, \\
		where $s$ is the number of open annuli, $d$  is the number of  open once-punctured disks, among connected components of $ S \backslash C $, $ t $ is the number of all boundary components with no marked point on their boundaries in $ (C, C\cap M)$ and $ n_{b} $ is the number of boundary arcs of $C$ which are tagged arc in $ (S,M) $.
	
\end{lem}
\begin{proof}

	If $C$ is either a once-punctured disk or $C\cap M \neq \emptyset$, there is no tagged arc lying in $C$. We assume  $C$ is different from these cases in the following proof.
	
	Suppose there are $d$ open once-punctured disks $D_{p_{i}}, i=1, ..., d$ among connected components of $ S \backslash C $ and each puncture point $ p_{i}  $ is in $ D_{p_{i}} $. For these tagged arcs connected to two of $p_{i} (i=1, ..., d)$, we do flips at these tagged arcs such that a tagged arc connected to two of $p_{i} (i=1, ..., d)$ is replaced with a tagged arc connected to at most one of $p_{i} (i=1, ..., d)$ at each flip.

	Since the least number of tagged arcs connected to $p_{i}$ but connected to none of $\{p_{j},j\neq i\}$ in a triangulation of $(S,M)$  is $ 2 $, the maximal tagged arcs sets of $ (C, C\cap M) $ are exactly the sets obtained from the maximal tagged arcs set of $(C\cup  \bigcup_{i} D_{p_{i}},(C\cup \bigcup_{i} D_{p_{i}})\cap M ) $ with exactly $2d$ tagged arcs connected to one of $p_{i} (i=1, ..., d)$ by removing these $2d$ tagged arcs.

	In the following, we always assume that no connected component of $ S \backslash C $ is an open once-punctured disk if there is no additional explanation.
	
	In fact, $ C $ is an oriented surface with marked points $ C\cap M $, but some boundary components may contain no marked point. Any triangulation $ T $ of $ (S, M) $ with a maximal number of tagged arcs lying in $ C $ must contain a tagged arc $ \gamma_{b} $ with the same endpoint for each boundary component $ b $ with no marked point such that $ \gamma_{b} $ lies in $ C $ and cuts out an open region $ U_{\gamma_{b}} $ satisfying that $ C\backslash U_{\gamma_{b}}$ contains all tagged arcs in $ T $ lying in $ C $.

Note that $ \gamma_{b_{1}} $ and $ \gamma_{b_{2}} $ may be the same for different boundary components $ b_{1} $, $ b_{2} $ (up to isotopy). In this case, $ C $ is a once-punctured cylinder or a connected component of $ S \backslash C $ is an open annulus.
	
	Thus removing these $ U_{\gamma_{b}} $, $ (C\backslash \bigcup_{b} U_{\gamma_{b}},(C\backslash \bigcup_{b} U_{\gamma_{b}}) \cap M) $ is a marked surface, where the union is over all boundary components with no marked point. And these tagged arcs lying in $ C $ still lie in $ (C\backslash \bigcup_{b} U_{\gamma_{b}},(C\backslash \bigcup_{b} U_{\gamma_{b}}) \cap M) $.
	
	Suppose for $ (C, C\cap M) $, the genus is $ g $, the number of boundary components is $ a $, the number of punctures is $ p $ and the number of marked points on the boundaries  is $ c $. Then for $ U_{\gamma_{b}} $,
	if the endpoint of $ \gamma_{b} $ is a puncture, then the genus and number of boundary components, punctures, marked points on the boundaries in $ (C\backslash U_{\gamma_{b}}, (C\backslash U_{\gamma_{b}})\cap M ) $ are $ g $, $ a-1+1 $, $ p-1 $ and $ c+1 $, respectively; if the endpoint of $ \gamma_{b} $ is a marked point on the boundaries, then the genus and number of boundary components, punctures, marked points in $ (C\backslash U_{\gamma_{b}}, (C\backslash U_{\gamma_{b}})\cap M ) $ are $ g $, $ a-1 $, $ p $ and $ c+1 $, respectively.
	
	Let $ n_{b} $ be the number of boundary arcs of $C$ which are in $ T $ and there are $s$ open annuli among connected components of $ S \backslash C $.
	
Note that if $(C\backslash \bigcup_{b} U_{\gamma_{b}},(C\backslash \bigcup_{b} U_{\gamma_{b}}) \cap M)$ is a disk with 1, 2 or 3 marked points on the boundaries, or the once-punctured disk with 1 marked point on the boundaries, $(C\backslash \bigcup_{b} U_{\gamma_{b}},(C\backslash \bigcup_{b} U_{\gamma_{b}}) \cap M)$ does not admit a triangulation.

Now we give a classification of $C$ as follows.

In the case $(C\backslash \bigcup_{b} U_{\gamma_{b}},(C\backslash \bigcup_{b} U_{\gamma_{b}}) \cap M)$ is a disk with $2$ marked points on the boundaries, $ C $ is (I) either (i) the digon, or (ii) the annulus with 1 marked point on the boundaries or (iii) the once-punctured annulus.

 In the case $(C\backslash \bigcup_{b} U_{\gamma_{b}},(C\backslash \bigcup_{b} U_{\gamma_{b}}) \cap M)$ is a disk with $3$ marked points on the boundaries, $ C $ is (II) either (i) the triangle, or (ii) the annulus with 2 marked points on the same boundary component or (iii) a pair of pants with 1 puncture.

In the case $(C\backslash \bigcup_{b} U_{\gamma_{b}},(C\backslash \bigcup_{b} U_{\gamma_{b}}) \cap M)$ is the once-punctured disk with 1 marked point on the boundaries, $ C $ is (III) either (i) the once-punctured disk with 1 marked point on the boundaries or (ii) the twice-punctured disk.

 Let $\mathfrak N$ denote a cardinality of a maximal tagged arcs set. Then respectively,\\
 $\mathfrak N=n_{b}$ for $C$ in the cases (I)(i), (I)(ii) and (II)(i); \\
  $\mathfrak N=1-d$ for $C$ in the cases (I)(iii) and (II)(ii); \\
  $\mathfrak N=3-d-s$ for $C$ in the case (II)(iii);\\
   $\frak N=2$ for $C$ in the case (III)(i)-(ii),\\
   where $d$ is the number of open once-punctured disks among connected components of $ S \backslash C $.
	
	According to Proposition \ref{narc}, the number of tagged arcs in $ T $ lying in $ (C\backslash \bigcup_{b} U_{\gamma_{b}},(C\backslash \bigcup_{b} U_{\gamma_{b}}) \cap M) $ except those $ \gamma_{b} $ is equal to $ n'=6g+3a+3p+c-6-2t+n_{b} $, where $ t $ is the number of all boundary components with no marked point in $ (C, C\cap M)$.
	
Note that $\gamma_{b}$ is counted twice for each pair of boundary components of an open annulus among connected components of $ S \backslash C $.
Thus the cardinality of a maximal tagged arcs set is equal to $$ \mathfrak{N} =n'+t-s=6g+3a+3p+c-6-2t+n_{b}+t-s=6g+3a+3p+c-6-t-s+n_{b}. $$ 
	
	Suppose $(C, C\cap M)$ is different from a connected subsurface in the case (I)-(IV) and there are $d$ open once-punctured disks among connected components of $ S \backslash C $, the cardinality of a maximal tagged arcs set of $ (C, C\cap M) $ is equal to $ \mathfrak{N}=6g+3a+3p+c-6-t-s-d+n_{b}  $.

Note that   this formula is also valid for those $C$ in the cases (I)(iii), (II)(i)-(iii) and (III)(ii).
	
\end{proof}

\begin{lem}\label{addarc-lem}
	Let $ (S, M) $ be a marked surface, $ C_{1}, C_{2} $ subsurfaces of $ (S, M) $. If the set of tagged arcs lying in $ C_{1} $ is a proper subset of that of $ C_{2} $ and boundary components of $C_{1}$ which contain exactly one marked point are contained in $\partial S$ and no connected component of $S\backslash C_{1}$ is a triangle with two edges in $\partial S$, then the cardinality of a maximal tagged arcs set of $ (C_{1}, C_{1}\cap M) $ is less than that of $ (C_{2}, C_{2}\cap M) $.

\end{lem}
\begin{proof}
	According to the condition, tagged arcs lying in a connected component of $ C_{1} $ also lies in a connected component of $ C_{2} $, and there are two connected components $C_{1}', C_{2}',$ of $ C_{1} $ and $C_{2}$, respectively, such that the set of tagged arcs lying in $ C_{1}' $ is a proper subset of that of $ C_{2}' $. Thus we can assume that $C_{1}, C_{2}$ is connected.

	Denote by $\mathfrak{N}_{1}, \mathfrak{N}_{2}  $ the  cardinalities of a maximal tagged arcs set of $ (C_{1}, C_{1}\cap M) $ and $ (C_{2}, C_{2}\cap M) $, respectively.
	If $C_{1}$ satisfies  one of the cases  (1)-(3) in Lemma \ref{subnarclem}, it is easy to check that this lemma holds. In the following, we always assume that $C_{1}$ is different from that in cases (1)-(3).
	
	Since  the set of tagged arcs lying in $ C_{1} $ is a proper subset of that of $ C_{2} $, at least  one of  genuses, numbers of boundary components, punctures and marked points on the boundaries of $(C_{1}, C_{1}\cap M)$ and $(C_{2}, C_{2}\cap M)$
	 must be different each other.
	Denote by $g_{i}, B_{i}, P_{i}, A_{i}$ the genus, sets of boundary components, punctures and marked points on the boundaries of $(C_{i}, C_{i}\cap M)$, respectively for $i=1, 2$.	
	In fact, we have $g_{1} \le g_{2}$, $P_{1} \subset P_{2}$,  otherwise we can find a proper tagged arc lies in $C_{1}$ but does not lie in $C_{2}$.
	
	Since $C_{1}$ is different from $C_{2}$, there must be $B_{1} \backslash B_{2} \neq \emptyset $.
	In these cases, some boundary component $b$ in $C_{1}$ vanishes via gluing some subset of the connected component of $S \backslash C_{1}$ bounded by boundaries with non-empty intersection with $b$ and some marked points in $ A_{1}$ become punctures in $P_{2}$.
	According to Lemma \ref{subnarclem}, $
	\mathfrak{N}_{2}$ is greater than $
	\mathfrak{N}_{1}$ except the following cases:
	
(1) $C_{2}$ is obtained from $C_{1}$ by gluing an annulus with one marked point on the boundaries along the boundary component with one marked point for some boundary component $b$ of $C_1$  with one marked point,  see Figure \ref{fig-case1}.

	(2) $C_{2}$ is obtained from $C_{1}$ by gluing some triangle whose two edges are boundary arcs along another edge, see Figure \ref{fig-case1}.
	
	Note that in cases (1) and (2) due to Lemma \ref{subnarclem},  $
		\mathfrak{N}_{1}=
		\mathfrak{N}_{2}$. 
		
	 Since boundary components of $C_{1}$ which contain exactly one marked point are contained in $\partial S$ and no connected component of $S\backslash C_{1}$ is a triangle with two edges in $\partial S$, $n_{2}$ is always greater than $n_{1}$.
	
\end{proof}

\begin{figure}[htbp]
	\centering
	\begin{minipage}[t]{0.5\linewidth}
		\centering
	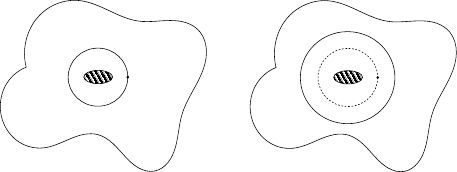
	\caption{Case (1)}
	\label{fig-case1}
	\end{minipage}%
	\begin{minipage}[t]{0.5\linewidth}
		\centering
	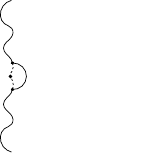
	\caption{Case (2)}
	\label{fig-case2}
	\end{minipage}%
	\centering
\end{figure}

\begin{thm}\label{c1}
	Let $ (S, M) $ be a feasible surface, $ \mathcal{A}(S, M) $ its cluster algebra without coefficient, and a subgroup $ H \le \text{\em Aut}\mathcal{A}(S, M) $. If $ \mathcal{A}(\Sigma_{i}), i=1, 2 $ are maximal in $ \mathcal{A}(S,M)^{H} $ as cluster subalgebras of $ \mathcal{A}(S,M) $, then the ranks of $ \mathcal{A}(\Sigma_{i}), i=1, 2 $ are the same.
\end{thm}
\begin{proof}
	If there is no exchange variable in mixing-type sub-seed in $ \mathcal{A}(S, M)^{H} $, then the ranks of $ \mathcal{A}(\Sigma_{i}), i=1, 2$ are  equal to zero.
	
	In the next proof, we assume  there is an exchange variable in mixing-type sub-seed in $ \mathcal{A}(S, M)^{H} $. Thus according to Remark \ref{extodirect}, $ H \le \text{Aut}^{+}\mathcal{A}(S, M) $.
	
	Suppose the oriented tagged arcs whose orientation are preserved under $ \rho(H) $ and correspond to cluster variables in $ \Sigma_{1} $ belong to a triangulation $ T $ of $ (S, M) $. According to Lemma \ref{tagstable}, we can assume all punctures are tagged plain except for some punctures tagged in different ways in $ T $.
	
	Due to Remark \ref{extodirect}, all oriented tagged arcs corresponding to exchange variables in $ \Sigma_{1} $ are preserved under $ \rho(H) $ and hence elements in $H$ is orientation-preserved.  According to Lemma \ref{symnoasymlem}, oriented tagged arcs corresponding to cluster variables in $\mathcal{A}(S,M)^{H}$ is invariant under $\rho(H)$.  Thus $ T $ contains all tagged arcs corresponding to cluster variables in $ \Sigma_{1} $.

 According to Lemma \ref{tagstable}, we can assume that each tagged arc is tagging plain in $T$ except some arcs tagging in different ways at some punctures which are endpoints of tagged arcs corresponding to cluster variables in $\Sigma_1$.

We choose some regions determined by cluster variables in $\Sigma_{1}$ as follows.

(1) Suppose  $ \gamma, \gamma^{p}$ correspond to cluster variables in $ \Sigma_{1} $. We choose the loop $\zeta$ such that $\gamma$ and $\zeta $ form a self-folding and denote by $U_{\zeta}$ the region bounded by $\zeta$;

(2) Suppose $ \gamma_{i}, i=1, 2, 3$ are ones of the following cases: a loop of a self-folding triangle such that there is a tagged arc in $T$ tagging notched at the interior puncture of this self-folding triangle, boundar arc of $S$ and a tagged arc  corresponding to a cluster variables in $ \Sigma_{1} $. If  $\gamma_{1}, \gamma_{2}, \gamma_{3} $  form a triangle, we denote by  $U_{\gamma_{1},\gamma_{2},\gamma_{3}} $ this simply-connected triangle.

(3) If $ x_{\gamma} \in \Sigma_{1} $ and $\gamma$ does not satisfy the two condition mentioned above, we denote by $ U_{\gamma} $ the digon in which $ \gamma $ lies.

	Let $ U_{0} $ be the union of   $ U_{X} $ over the regions we chosen, then $ U_{0} $ contains all tagged arcs corresponding to cluster variables in $ \Sigma_{1} $. Then each $ U_{X} $ is invariant under $ \rho(H)$, i.e., for each $\varphi \in  \rho(H) $, there is  $\varphi'$ isotopic to $\varphi$ such that $\varphi'(u)=u, u \in U_{X}$. Hence $ U_{0} $ is also invariant under $ \rho(H) $.  Suppose either (i) $\gamma_{1}, \gamma_{2}$ corresponding to cluster variables in $\Sigma_{1}$ or (ii) one of them corresponding to cluster variables in $\Sigma_{1}$ and the other is boundary arc of $S$ such that $\gamma_{1} $ is adjacent to $\gamma_{2}$. Since each $f \in H$ is orientation-preserved, for some $\gamma_{3}$ in $T$ that is adjacent to $\gamma_{1}$ and $\gamma_{2}$, $\gamma_{3}$ is invariant under $\rho(H)$. It means that $b \backslash M$ is connected for each boundary component $b$.
	
	Let  $\{b_{1}, ..., b_{h}\}$ be the set consisting of boundary components of $U_{0}$ which have one marked point and  are not contained in $\partial S$.
     Suppose $U$ is obtained by gluing an annulus with one marked point on the boundaries along the boundary component with one marked point for each $b_{i}$.
      Thus $U$ is a subsurface of $(S,M)$.
    Since $\rho(f), f\in H$ maps these annuli to themselves and  its restriction to each annulus is a power of the Dehn twist about the boundary, $U$ is invariant under $ \rho(H)$.

    In the following proof, we show that $ x_{\gamma} \in \mathcal{A}(S,M)^{H} $
		if and only if the corresponding tagged arc $ \gamma $ lies in $ U $ (up to isotopy).
	
	 Otherwise, suppose there is an oriented tagged arc $ \eta $ invariant under $ \rho(H) $ not lying in $ U $. According to Lemma \ref{tagstable} we can assume $ \eta $ is an arc tagged plain. Similarly, we can choose a simply-connected digon $U_{\eta}$ containing $\eta$ which is invariant under $ \rho(H) $.
	
	Denote by $ U' $ the union of $ U_{\eta} $ and $ U $. Suppose $\eta$ intersects the boundary components of $U$ with minimal number of intersection points. Since $U$ is invariant under $ \rho(H) $, the image of each intersection point lies in the same segment of boundary component (tagged arc) of $U$ and  the order of intersection points is preserved under $ \rho(H) $. Since the connected components of the closure of $U_{\eta} \backslash U$ are simply-connected regions and $ U_{\eta}$ is invariant under $ \rho(H) $, the connected components of the closure of $U_{\eta} \backslash U$ are also invariant under $ \rho(H) $. Thus $U'$ is invariant under $ \rho(H) $.

	Note that if $U'$ is not a subsurface of $(S,M)$, we can similarly gluing annuli to obtain a subsurface, which is invariant under $ \rho(H) $. Thus we assume $U'$ is a subsurface for convenience.
	
	It follows from Lemma \ref{subnarclem} that the cardinality of a maximal tagged arcs set of $ U $ (resp.  $ U' $) is constant, which is denoted by $ \mathfrak{N} $ (resp. $ \mathfrak{N}' $).
	And it follows from Lemma \ref{addarc-lem} we have $ \mathfrak{N}' > \mathfrak{N} $. Then we can find a maximal  tagged arc set of $ U' $ which contains tagged arcs whose corresponding cluster variables in $ \Sigma_{1} $.
	It means we can find a mixing-type sub-seed $ \Sigma' \subset \mathcal{A}(S,M)^{H} $ containing $ \Sigma_{1} $, which is contrary to that $ \mathcal{A}(\Sigma_{1}) $ is a maximal cluster subalgebra in $ \mathcal{A}(S,M)^{H} $.
	
	Thus the tagged arcs corresponding cluster variables in $ \Sigma_{2} $ lie in $ U $ and hence the rank of $ \mathcal{A}(\Sigma_{2}) $ is less than or equal to the rank of $ \mathcal{A}(\Sigma_{1}) $.
	
	Similarly, we can show the rank of $ \mathcal{A}(\Sigma_{1}) $ is less than or equal to the rank of $ \mathcal{A}(\Sigma_{2}) $.
	Hence the rank of $ \mathcal{A}(\Sigma_{2}) $ is the same as $ \mathcal{A}(\Sigma_{1}) $'s.
	
\end{proof}

\begin{defi}\label{def-maxsubsurface}
	 	Let $ (S,M) $ be a feasible surface, $ \mathcal{A}(S,M) $ its cluster algebra without coefficient,  $ H \le \text{\em Aut} \mathcal{A}(S,M) $ and  any $ f \in H $ is orientation-preserved. A subsurface   $ U $ of $ (S,M) $ is called a \textbf{maximal invariant subsurface} of $ H $ in case $ x_{\gamma} \in \mathcal{A}(S,M)^{H} $ if and only if the corresponding tagged arc $ \gamma $ lies in $ U $ (up to isotopy).
	
\end{defi}

\begin{rmk}\label{maxsubsurfaceproperty-rmk}
	Suppose $ U_{1}, U_{2} $ are maximal invariant subsurfaces of $ H $, their boundary components including exactly one marked point are contained in $\partial S$ and  all tagged arcs are contained in their connected components. Then the genuses, the sets of boundary components, the punctures and marked points on the boundaries of $ U_{1} $ and $ U_{2} $ are the same respectively. Otherwise we can find a proper tagged arc which lies in one of $U_{1}$ and $U_{2}$ but does not lie in the other, which is contrary to Definition \ref{def-maxsubsurface}. Thus in the sequel, we always assume maximal invariant subsurfaces satisfying the above conditions.
	
	Note that the proof of Theorem \ref{c1} provides a way to obtain a maximal invariant subsurface $ U $ of $ H $ for the case that any $  f \in  H $ is orientation-preserved.
	Suppose $ \mathcal{A}(\Sigma_{1}) $ is a maximal cluster subalgebra in $\mathcal{A}(S,M)^{H}$. Similar to the proof of Theorem \ref{c1}, we can prove that for an orientation-preserved cluster automorphism $ f $, $ \Sigma_{1} $ is invariant under $ f $ if and only if $ U $ is invariant under $ \rho(f) $. And the set of tagged arcs corresponding to cluster variables in $ \Sigma' $ is a maximal tagged arcs set of $ (U, U\cap M) $ if and only if $ \mathcal{A}(\Sigma') $ is a maximal cluster subalgebra in $\mathcal{A}(S,M)^{H}$.
	
\end{rmk}

For cluster algebras from feasible surfaces, Conjecture \ref{srk} also holds for maximal cluster subalgebras $ \mathcal{A}(\Sigma_{i}), i=1, 2 $ in $ \mathcal{A}(S,M)^{H} $, even though they are not Galois-like extension subalgebra. In fact, we have the following corollary.

\begin{cor}\label{surfacegaloislikecor}
	Let $ (S,M) $ be a feasible surface, $ \mathcal{A}(S,M) $ its cluster algebra without coefficient, $ H \in \mathit{2}^{\text{\em Aut}\mathcal{A}} \backslash \text{ker} \phi $ whose  any element  $  f $ is orientation-preserved.  Then $ \mathcal{A}(S,M) $ is a Galois-like extension over $ \mathcal{A}(\Sigma_{1}) $ for any maximal cluster subalgebra $ \mathcal{A}(\Sigma_{1}) $ in $ \mathcal{A}(S,M)^{H} $.
\end{cor}
\begin{proof}
	Let $ U $ be the maximal invariant subsurface of $ H $ and $ \mathcal{A}(\Sigma') \in \mathcal{M}_{sub}^{H} $.

	If the connected components of $U$ are digons, once-punctured annuli or twice-punctured disks, then the maximal tagged arcs set of $U$ uniquely (neglecting the taggings) consists of two edges of  each digon, the loop lying in each annulus and two tagged arcs tagged in different ways.   Hence $\text{Gal}_{\Sigma_{1}}\mathcal{A}(S,M)=\text{Gal}_{\Sigma'}\mathcal{A}(S,M)$ by Lemma \ref{tagstable}.
	
	Otherwise, in other cases, for any maximal tagged arcs $A$ set of $U$, after adding boundary arcs of $S$ and replacing each pair tagged arcs in $A$ which is the form as $\gamma, \gamma^{p}$ with the self-folding triangle containing $\gamma$, there exists a triple of tagged arcs in the new set forming a triangle. According to Remark \ref{maxsubsurfaceproperty-rmk}, the set of taged arcs corresponding to $\Sigma_{1}$ is a maximal tagged arcs set $A_{1}$ of $U$. Thus after modifying, there exists a triple of tagged arcs in the new set obtained from $A_{1}$ such that this triple forms a triangle. Besides, cluster variables corresponding to elements in this new set are invariant under $\text{Gal}_{\Sigma_{1}}\mathcal{A}(S,M)$ and the existence of a triangle guarantee that elements in $\text{Gal}_{\Sigma_{1}}\mathcal{A}(S,M)$ are orientation-preserved.
	
	Since  elements in $H$ are orientation-preserved, according to Remark \ref{maxsubsurfaceproperty-rmk}, $f \in \text{Gal}_{\Sigma_{1}}\mathcal{A}(S,M)$ if and only if $U$ is invariant under $\rho(f)$, which is equivalent to $f \in \text{Gal}_{\Sigma'}\mathcal{A}(S,M)=H$.
	
\end{proof}

Now we consider the Galois inverse problem for a feasible surface $ (S, M) $, that is,  under what condition is a subgroup $H$ of $\text{Aut}\mathcal{A}(S, M) $ a Galois group?

 Let $(C,C\cap M)$ be a subsurface of the marked surface $(S,M)$.
	Denote by $ \text{Mod}(C,C\cap M) $ (resp. $ \text{Mod}_{\bowtie}(C,C\cap M) $ the group consiting of elements of $ \mathcal{MCG}(C,C\cap M) $ (resp. $ \mathcal{MCG}_{\bowtie}(C,C\cap M) $) fixing boundary components of $C$ not contained in $\partial S$  if $S$ is a once-punctured closed surface $S$ (resp. if $S$ is not a once-punctured closed surface).

	Note that for a feasible surface $ (S,M) $ and a subsurface $ C $ of $ S$,
	there is a natural homomorphism $ \iota : \text{Mod}_{\bowtie}(C,C\cap M) \to \mathcal{MCG}_{\bowtie}(S,M) $ induced by the embedding.
	According to \cite{FM12}, if no connected component of $ S \backslash C $ is an open annulus, an open disk, or an open once-punctured disk, then $  \iota $ is a monomorphism. In this case,  $\text{Mod}(C,C\cap M)$ (resp. $ \text{Mod}_{\bowtie}(C,C\cap M) $) is a subgroup of $ \mathcal{MCG}(S,M) $ (resp. $ \mathcal{MCG}_{\bowtie}(S,M)$) under the monomorphism.

\begin{prop}\label{prop-Galois-inverse-problem}
	Assume $ \mathcal{A}(S, M) $ is a cluster algebra without coefficient from a feasible surface $(S, M)$. Let $ C $ be a subsurface of $ S $ such that there is  no connected component of  $ U= \overline{S\backslash C}  $ which is a closed annulus or a closed once-punctured disk and there is a connected component of $U$ which i different from a digon, a closed once-punctured annulus and a closed twice-punctured disk, where $\overline{S\backslash C}$ means the closure of $S\backslash C$. Then we have the following statements:
	
	(a) \; If $ (S, M) $ is a once-punctured closed surface, then
	
	(i)\; $ \rho'^{-1}(\text{\em Mod}(C,C\cap M)) $ is a Galois group.
	
	(ii)\;	Conversely, for a Galois group $ H  $ whose any element $ f $ is orientation-preserved, there is a subsurface $C$ such that $ H=\rho'^{-1}(\text{\em Mod}(C,C\cap M)) $.
	
	(b) \; If $ (S, M) $ is not a once-punctured closed surface, then
	
		(i)\; $ \rho'^{-1}(\text{\em Mod}_{\bowtie}(C,C\cap M)) $ is a Galois group.
	
	(ii)\;	Conversely, for a Galois group $ H  $ whose any element $ f $ is orientation-preserved, there is a subsurface $C$ such that $ H=\rho'^{-1}(\text{\em Mod}_{\bowtie}(C,C\cap M)) $.
\end{prop}
\begin{proof}
We only prove (b). The proof of (a) is similar.
	
	Since no connected component of  $ U  $ is a closed annulus, a closed disk, or a closed once-punctured disk, the homomorphism $$ \iota:\; \text{Mod}_{\bowtie}(C,C\cap M) \to \mathcal{MCG}_{\bowtie}(S,M) $$ induced by the embedding is monomorphic. Hence we have $ G=\rho'^{-1}(\text{Mod}_{\bowtie}(C,C\cap M)) \le \text{Aut}\mathcal{A}(S,M)$.
	
	Since there is a connected component of $U$ which is different from a digon, a closed once-punctured annulus and a closed twice-punctured disk,  we have that the elements in $G$ are orientation-preserved. Let $U'$ be the maximal invariant subsurface of $G$, we can assume $U \subset U'$. Similar to the proof of Corollary \ref{surfacegaloislikecor}, we can show that for any maximal cluster subalgebra $\mathcal{A}(\Sigma')$ in $\mathcal{A}(S,M)^{G} $, the elements in $\text{Gal}_{\Sigma'}\mathcal{A}(S,M)$ are orientation-preserved.
	
	Let
	\[ \mathcal{MCG}_{\bowtie}(S,M)_{X}= \{\varphi \in  \mathcal{MCG}_{\bowtie}(S,M) \;| \; X \text{ is invariant under a representative of } \varphi \}. \]
	According to Remark \ref{maxsubsurfaceproperty-rmk}, if $f \in \text{Gal}_{\Sigma'}\mathcal{A}(S,M) $, then
 $$\rho'(f) \in \mathcal{MCG}_{\bowtie}(S,M)_{U'} \subset \mathcal{MCG}_{\bowtie}(S,M)_{U}=\rho'(G).$$
 Thus $ \text{Gal}_{\Sigma'}\mathcal{A}(S,M) \le G$. On the other hand, it is obvious that $\text{Gal}_{\Sigma'}\mathcal{A}(S,M) \ge G$ and hence we have $\text{Gal}_{\Sigma'}\mathcal{A}(S,M) = G$.
	
	Conversely, suppose $ H=\text{Gal}_{\Sigma'}\mathcal{A}(S,M)$ and its elements are orientation-preserved. Let $ U $ be the maximal invariant subsurface of $ H $ and $ C=\overline{S\backslash U} $. Since $ \overline{S\backslash C} =U$ and  no connected component of  $ U $ is a closed annulus, a closed disk, or a closed once-punctured disk, we have $ \rho'^{-1}(\text{Mod}_{\bowtie}(C,C\cap M)) \le \text{Aut}\mathcal{A}(S,M) $.
	
 According to Remark \ref{maxsubsurfaceproperty-rmk}, since $ f \in H $ or $ f \in \rho'^{-1}(\mathcal{MCG}_{\bowtie}(S,M)_{U}) $ is orientation-preserved,  we have  $ H=\text{Gal}_{\Sigma'}\mathcal{A}(S,M) =\rho'^{-1}(\mathcal{MCG}_{\bowtie}(S,M)_{U}) = \rho'^{-1}(\text{Mod}_{\bowtie}(C,C\cap M)) $.
	
\end{proof}

Note that the maximality of $ \mathcal{A}(\Sigma_{i}), i=1, 2 $ in $ \mathcal{A}(S,M)^{H} $ does not guarantee that $ \mathcal{A}(\Sigma_{1})=\mathcal{A}(\Sigma_{2}) $.

\begin{figure}[H]
	\centerline{\includegraphics[scale =0.2]{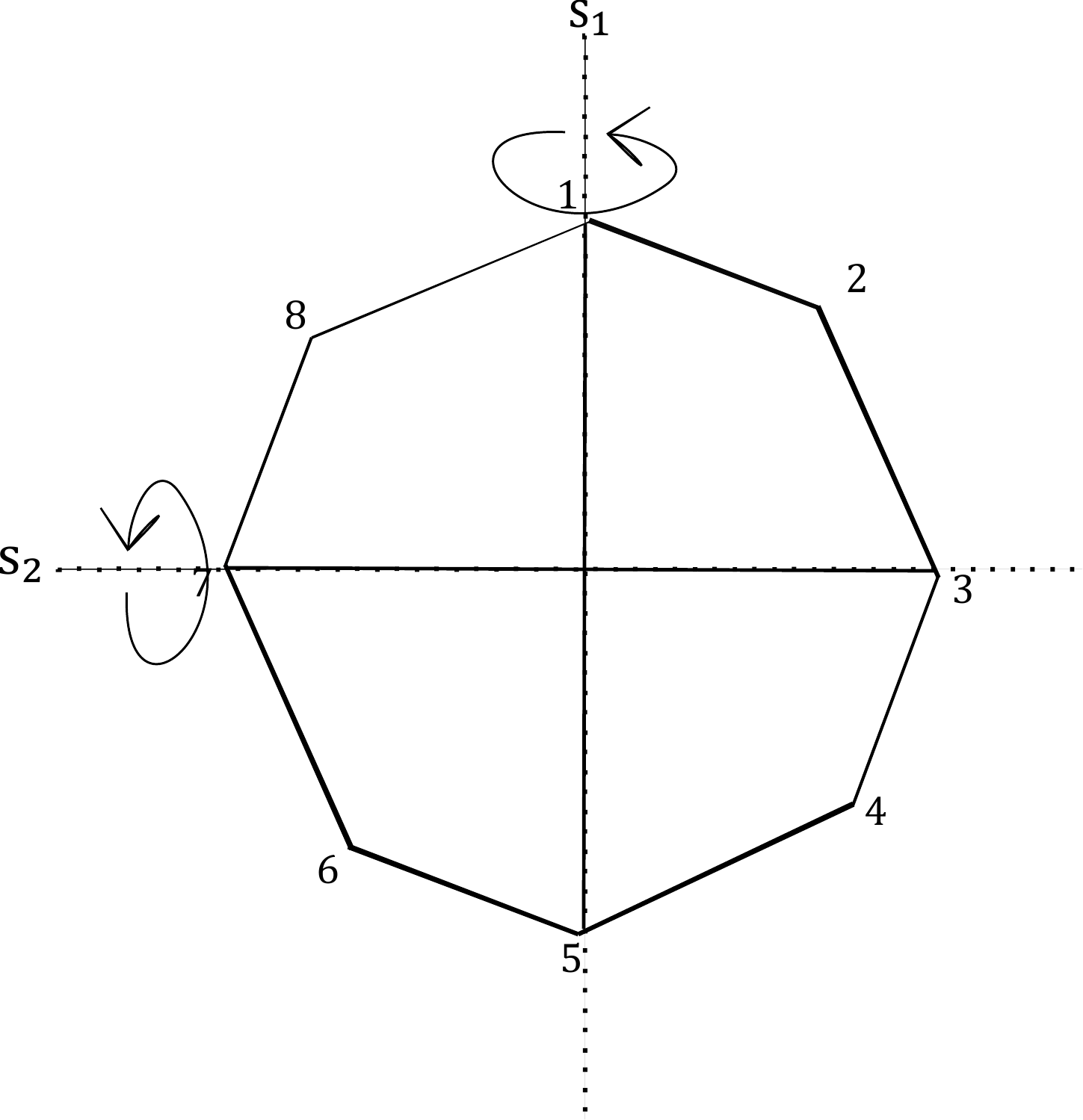}}
	\caption{Reflections $ s_{1}, s_{2} $ in octagon} \label{octa}
\end{figure}

\begin{ex}\label{ex8gon}
	Suppose $ P $ is an octagon whose vertices are labelled by $ 1, ..., 8 $ clockwise see Figure \ref{octa}, $ \mathcal{A}(\Sigma_{1}) $ (resp. $ \mathcal{A}(\Sigma_{2}) $) is the cluster subalgebra with frozen variables corresponding to the diagonal $ \gamma_{15} $ (resp. $ \gamma_{37} $), $ H $ is the subgroup generated by reflections $ s_{1}, s_{2} $ and the $ 180^{\circ} $ rotation. Then $ \mathcal{A}(\Sigma)_{i} \in \mathcal{M}_{sub}^{H}, i=1, 2 $, but $\mathcal{A}(\Sigma_{1}) \neq \mathcal{A}(\Sigma_{2}) $.
\end{ex}

\begin{thm}\label{c2}
	Let $ \mathcal{A}(S,M) $ be a cluster algebra from a feasible surface $(S,M)$ without coefficient, $ H \le \text{\em Aut}\mathcal{A}(S,M) $, $ \mathcal{A}(\Sigma_{1}) $, $ \mathcal{A}(\Sigma_{2}) \in \mathcal{M}^{H}_{sub} $,
	$ H_{1}, H_{2},...,H_{s} \in \mathit{2}^{\text{\em Aut}\mathcal{A}(S,M)} \backslash \text{ker}\phi $ and $ H_{1} \le H_{2} \le \cdots \le H_s$. Then for any $ \mathcal{A}(\Sigma_{i}) \in \mathcal{M}^{H_{i}}_{sub}, 1 \leq i \leq s $, there exist $ \mathcal{A}(\Sigma_{j}) \in \mathcal{M}^{H_{j}}_{sub}$ for $ j=1,..., i-1, i+1, ..., s $ such that
	\begin{equation}\label{Aseq}
		\mathcal{A}(\Sigma_{1}^{(d)}) \ge \cdots\ge \mathcal{A}(\Sigma_{i}^{(d)}) \ge \cdots \ge \mathcal{A}(\Sigma_{s}^{(d)}).
	\end{equation}
	That is, there is a Galois map for $ \mathcal{A}(S, M) $.
	
	In particular, if $ H_{1} < H_{2} <\cdots < H_{s} $ and any $ f \in H_{s} $ is orientation-preserved, then we have
	\begin{equation}\label{Aseq1}
		\mathcal{A}(\Sigma_{1}) > \cdots > \mathcal{A}(\Sigma_{i}) > \cdots > \mathcal{A}(\Sigma_{s}).
	\end{equation}
\end{thm}
\begin{proof}
	To prove (\ref{Aseq}) holds for any $ \mathcal{A}(\Sigma_{i}) \in \mathcal{M}^{H_{i}}_{sub}, 1 \leq i \leq s $, it is sufficient to prove that for any $ \mathcal{A}(\Sigma_{i}) \in \mathcal{M}^{H_{i}}_{sub} $, there exist $ \mathcal{A}(\Sigma_{j}) \in \mathcal{M}^{H_{j}}_{sub} $ for $ j= i+1, i-1 $ such that
	\begin{equation}\label{Asseq}
		\mathcal{A}(\Sigma_{i-1}^{(d)}) \ge \mathcal{A}(\Sigma_{i}^{(d)}) \ge \mathcal{A}(\Sigma_{i+1}^{(d)}).
	\end{equation}

	Suppose $ t $ is the maximal index such that any $ f \in H_{t} $ is orientation-preserved, if the condition is false for any $ i\in [1,s] $,
	we set $ t=0 $. If $ i > t $, according to Remark \ref{extodirect}, $ \Sigma_{i}^{(d)} $ is trivial for any $ \mathcal{A}(\Sigma_{i}) \in \mathcal{M}^{H_{i}}_{sub} $ and hence (\ref{Asseq}) holds.

	If $ i \le t $, there is a maxiaml $ \mathcal{A}(\Sigma_{i-1}) \in \mathcal{A}(S,M)^{H_{i-1}} $ such that $ \mathcal{A}(\Sigma_{i-1}) \ge \mathcal{A}(\Sigma_{i}) $ and $ f $ is orientation-preserved for any $ f \in H_i $.
	Since $ H_{i-1} \in \mathit{2}^{\text{Aut}\mathcal{A}(S,M)} \backslash \text{ker}\phi $, according to Corollary \ref{surfacegaloislikecor},
	we have $ \mathcal{A}(\Sigma_{i-1}) \in \mathcal{M}^{H_{i-1}}_{sub}$.
	
	On the other hand, if $ i\ge t $, then $ \Sigma_{i+1}^{(d)} $ is trivial for any $ \mathcal{A}(\Sigma_{i+1}) \in \mathcal{M}^{H_{i+1}}_{sub} $ and we have $ \mathcal{A}(\Sigma_{i+1}^{(d)}) \le \mathcal{A}(\Sigma_{i}^{(d)}) $.
	Otherwise, $ H_{i} \leq H_{i+1} \leq \text{Aut}^{+}\mathcal{A}(S,M)$ and $ f $ is orientation-preserved for any $ f \in H_{i+1} $.
	
Since $  \mathcal{A}(S,M)^{H_{i+1}} \subset \mathcal{A}(S,M)^{H_{i}} $, according to Remark \ref{maxsubsurfaceproperty-rmk}, tagged arcs lying in the maximal invariant subsurface $\hat{U}_{i+1}$ of $ H_{i+1} $ also lie in the maximal invariant subsurface $\hat{U}_{i}$ of $ H_{i} $ (up to isotopy). Thus there are $\varphi_{i}, \varphi_{i+1} \in \text{Homeo}^{+}_{0}(S,M)$ such that $ U_{i+1}=\varphi_{i+1}(\hat{U}_{i+1}) \subset U_{i}=\varphi_{i}(\hat{U}_{i})$. Note that $U_{i}$ (resp. $U_{i}$) is also the maximal invariant subsurface of $H_{i}$ (resp. $H_{i+1}$). According to the proof of Theorem \ref{c1}, we can obtain $U_{i}$ from $U_{i}'$, which is the union of some areas determined by tagged arcs corresponding to cluster variables in $\Sigma_{i}$, by gluing annuli for all boundary components with one marked point. In this way, the cardinality of a maximal tagged arcs set is unchanged. Thus $U_{i}'$ contains a  maximal tagged arcs set of $U_{i+1}$ and it corresponding to a mixing-type sub-seed $\Sigma_{i+1}$. According to Remark \ref{maxsubsurfaceproperty-rmk}, $\mathcal{A}(\Sigma_{i+1}) \le \mathcal{A}(\Sigma_{i})$ and according to Corollary \ref{surfacegaloislikecor}, $\mathcal{A}(\Sigma_{i+1}) \in \mathcal{M}_{sub}^{H_{i+1}}$.
	Thus (\ref{Asseq}) holds, and hence (\ref{Aseq}) holds.
	
	If $ H_{1} < H_{2},...,H_{s}  $ and any $ f \in H_{s} $ is orientation-preserved, then $ t=s $ and we can see that (\ref{Asseq}) can be rewritten as \[ \mathcal{A}(\Sigma_{i-1}) > \mathcal{A}(\Sigma_{i}) > \mathcal{A}(\Sigma_{i+1}) ,\]
	since $ \mathcal{M}_{sub}^{H_{i}} \cap \mathcal{M}_{sub}^{H_{i+1}} = \emptyset $ for $ H_{i} \neq H_{i+1} $.
	Thus (\ref{Aseq1}) holds.
	
\end{proof}

\begin{cor}
	Let $ \mathcal{A}(S, M) $ be a cluster algebra from a feasible surface $(S, M)$ without coefficient, Suppose there is a strictly ascending sequence of Galois-like extension subalgebras of $ \mathcal{A}(S, M) $:
	\begin{equation}\label{subgroups} \mathcal{A}(\Sigma_{1}) > \cdots > \mathcal{A}(\Sigma_{i}) > \cdots > \mathcal{A}(\Sigma_{s}), \end{equation}
	and any $ f \in \text{\em Gal}_{\Sigma_{s}}\mathcal{A}(S,M) $ is orientation-preserved.
	Then a strictly descending sequence of subgroups of $ \text{\em Aut}\mathcal{A}(S,M) $:
	\begin{equation}\label{subalgebras}
		\text{\em Gal}_{\Sigma_{1}}\mathcal{A}(S,M)< \cdots <\text{\em Gal}_{\Sigma_{i}}\mathcal{A}(S,M) < \cdots <\text{\em Gal}_{\Sigma_{s}}\mathcal{A}(S,M),
	\end{equation}
	which is induced by Proposition \ref{strictgsub} and
	the sequence (\ref{subalgebras}) of cluster subalgebras cannot be refined if and only if the sequence (\ref{subgroups}) of subgroups cannot be refined.
\end{cor}
\begin{proof}
	According to Proposition \ref{strictgsub}, if the sequence (\ref{subalgebras}) of cluster subalgebras can be refined, then the sequence (\ref{subgroups}) of subgroups can be refined.
	
	Conversely, according to the last part of Theorem \ref{c2}, if the sequence (\ref{subalgebras}) of subgroups can be refined, then the sequence (\ref{subalgebras}) of cluster subalgebras can be refined.
\end{proof}

{\bf Acknowledgements.}\;
{\em This project is supported by the National Natural Science Foundation of China (No.12071422, No.12131015).}


\end{document}

%% file: f3.pdf_tex
\begingroup%
  \makeatletter%
  \providecommand\color[2][]{%
    \errmessage{(Inkscape) Color is used for the text in Inkscape, but the package 'color.sty' is not loaded}%
    \renewcommand\color[2][]{}%
  }%
  \providecommand\transparent[1]{%
    \errmessage{(Inkscape) Transparency is used (non-zero) for the text in Inkscape, but the package 'transparent.sty' is not loaded}%
    \renewcommand\transparent[1]{}%
  }%
  \providecommand\rotatebox[2]{#2}%
  \newcommand*\fsize{\dimexpr\f@size pt\relax}%
  \newcommand*\lineheight[1]{\fontsize{\fsize}{#1\fsize}\selectfont}%
  \ifx\svgwidth\undefined%
    \setlength{\unitlength}{472.77397896bp}%
    \ifx\svgscale\undefined%
      \relax%
    \else%
      \setlength{\unitlength}{\unitlength * \real{\svgscale}}%
    \fi%
  \else%
    \setlength{\unitlength}{\svgwidth}%
  \fi%
  \global\let\svgwidth\undefined%
  \global\let\svgscale\undefined%
  \makeatother%
  \begin{picture}(1,0.19131763)%
    \lineheight{1}%
    \setlength\tabcolsep{0pt}%
    \put(0,0){\includegraphics[width=\textwidth,page=1]{f3.pdf}}%
    \put(0.8,0.09){\color[rgb]{0,0,0}\makebox(0,0)[lt]{\lineheight{1.25}\smash{\begin{tabular}[t]{l}$ \frac{1}{3!}\sum_{\sigma \in S_{3}}\sigma $ \end{tabular}}}}%
  \end{picture}%
\endgroup%

%% file: smooth.pdf_tex
\begingroup%
  \makeatletter%
  \providecommand\color[2][]{%
    \errmessage{(Inkscape) Color is used for the text in Inkscape, but the package 'color.sty' is not loaded}%
    \renewcommand\color[2][]{}%
  }%
  \providecommand\transparent[1]{%
    \errmessage{(Inkscape) Transparency is used (non-zero) for the text in Inkscape, but the package 'transparent.sty' is not loaded}%
    \renewcommand\transparent[1]{}%
  }%
  \providecommand\rotatebox[2]{#2}%
  \newcommand*\fsize{\dimexpr\f@size pt\relax}%
  \newcommand*\lineheight[1]{\fontsize{\fsize}{#1\fsize}\selectfont}%
  \ifx\svgwidth\undefined%
    \setlength{\unitlength}{767.70497239bp}%
    \ifx\svgscale\undefined%
      \relax%
    \else%
      \setlength{\unitlength}{\unitlength * \real{\svgscale}}%
    \fi%
  \else%
    \setlength{\unitlength}{\svgwidth}%
  \fi%
  \global\let\svgwidth\undefined%
  \global\let\svgscale\undefined%
  \makeatother%
  \begin{picture}(0.4,0.108)%
    \lineheight{1}%
    \setlength\tabcolsep{0pt}%
    \put(0,0){\includegraphics[width=0.4\unitlength,page=1]{smooth.pdf}}%
    \put(0.098,0.045){\color[rgb]{0,0,0}\makebox(0,0)[lt]{\lineheight{1.25}\smash{\begin{tabular}[t]{l}$\alpha$\end{tabular}}}}%
    \put(0.0532,0.087){\color[rgb]{0,0,0}\makebox(0,0)[lt]{\lineheight{1.25}\smash{\begin{tabular}[t]{l}$\zeta$\end{tabular}}}}%
    \put(0.124,0.05){\color[rgb]{0,0,0}\makebox(0,0)[lt]{\lineheight{1.25}\smash{\begin{tabular}[t]{l}$=$\end{tabular}}}}%
    \put(0.27,0.05){\color[rgb]{0,0,0}\makebox(0,0)[lt]{\lineheight{1.25}\smash{\begin{tabular}[t]{l}$+$\end{tabular}}}}%
    \put(0.2,0.089){\color[rgb]{0,0,0}\makebox(0,0)[lt]{\lineheight{1.25}\smash{\begin{tabular}[t]{l}$\alpha_1$\end{tabular}}}}%
    \put(0.344,0.015){\color[rgb]{0,0,0}\makebox(0,0)[lt]{\lineheight{1.25}\smash{\begin{tabular}[t]{l}$\alpha_2$\end{tabular}}}}%
  \end{picture}%
\endgroup%

%% file: case1.pdf_tex
\begingroup%
  \makeatletter%
  \providecommand\color[2][]{%
    \errmessage{(Inkscape) Color is used for the text in Inkscape, but the package 'color.sty' is not loaded}%
    \renewcommand\color[2][]{}%
  }%
  \providecommand\transparent[1]{%
    \errmessage{(Inkscape) Transparency is used (non-zero) for the text in Inkscape, but the package 'transparent.sty' is not loaded}%
    \renewcommand\transparent[1]{}%
  }%
  \providecommand\rotatebox[2]{#2}%
  \newcommand*\fsize{\dimexpr\f@size pt\relax}%
  \newcommand*\lineheight[1]{\fontsize{\fsize}{#1\fsize}\selectfont}%
  \ifx\svgwidth\undefined%
    \setlength{\unitlength}{219.18875146bp}%
    \ifx\svgscale\undefined%
      \relax%
    \else%
      \setlength{\unitlength}{\unitlength * \real{\svgscale}}%
    \fi%
  \else%
    \setlength{\unitlength}{\svgwidth}%
  \fi%
  \global\let\svgwidth\undefined%
  \global\let\svgscale\undefined%
  \makeatother%
  \begin{picture}(1,0.37663651)%
    \lineheight{1}%
    \setlength\tabcolsep{0pt}%
    \put(0,0){\includegraphics[width=\unitlength,page=1]{case1.pdf}}%
    \put(0.47036795,0.18017877){\color[rgb]{0,0,0}\makebox(0,0)[lt]{\lineheight{1.25}\smash{\begin{tabular}[t]{l}$\Longrightarrow$\end{tabular}}}}%
    \put(0.16423992,0.22774845){\color[rgb]{0,0,0}\makebox(0,0)[lt]{\lineheight{1.25}\smash{\begin{tabular}[t]{l}$C_1$\end{tabular}}}}%
    \put(0.70151807,0.24706621){\color[rgb]{0,0,0}\makebox(0,0)[lt]{\lineheight{1.25}\smash{\begin{tabular}[t]{l}$C_2$\end{tabular}}}}%
  \end{picture}%
\endgroup%

%% file: case2.pdf_tex
\begingroup%
  \makeatletter%
  \providecommand\color[2][]{%
    \errmessage{(Inkscape) Color is used for the text in Inkscape, but the package 'color.sty' is not loaded}%
    \renewcommand\color[2][]{}%
  }%
  \providecommand\transparent[1]{%
    \errmessage{(Inkscape) Transparency is used (non-zero) for the text in Inkscape, but the package 'transparent.sty' is not loaded}%
    \renewcommand\transparent[1]{}%
  }%
  \providecommand\rotatebox[2]{#2}%
  \newcommand*\fsize{\dimexpr\f@size pt\relax}%
  \newcommand*\lineheight[1]{\fontsize{\fsize}{#1\fsize}\selectfont}%
  \ifx\svgwidth\undefined%
    \setlength{\unitlength}{74.22804465bp}%
    \ifx\svgscale\undefined%
      \relax%
    \else%
      \setlength{\unitlength}{\unitlength * \real{\svgscale}}%
    \fi%
  \else%
    \setlength{\unitlength}{\svgwidth}%
  \fi%
  \global\let\svgwidth\undefined%
  \global\let\svgscale\undefined%
  \makeatother%
  \begin{picture}(1,0.98684268)%
    \lineheight{1}%
    \setlength\tabcolsep{0pt}%
    \put(0.47435467,0.46945141){\color[rgb]{0,0,0}\makebox(0,0)[lt]{\lineheight{1.25}\smash{\begin{tabular}[t]{l}$\Longrightarrow$\end{tabular}}}}%
    \put(0.11406921,0.61036793){\color[rgb]{0,0,0}\makebox(0,0)[lt]{\lineheight{1.25}\smash{\begin{tabular}[t]{l}$C_1$\end{tabular}}}}%
    \put(0,0){\includegraphics[width=\unitlength,page=1]{case2.pdf}}%
    \put(0.86656174,0.60901837){\color[rgb]{0,0,0}\makebox(0,0)[lt]{\lineheight{1.25}\smash{\begin{tabular}[t]{l}$C_2$\end{tabular}}}}%
    \put(0,0){\includegraphics[width=\unitlength,page=2]{case2.pdf}}%
  \end{picture}%
\endgroup%